\documentclass[reqno,12pt,letterpaper]{amsart}
\usepackage{amsmath,amssymb,amsthm,graphicx,mathrsfs,url}
\usepackage[usenames,dvipsnames]{color}
\usepackage[colorlinks=true,linkcolor=Red,citecolor=Green]{hyperref}
\usepackage[usenames,dvipsnames]{xcolor}
\usepackage{tikz}
\usepackage{graphicx}
\usepackage[colorlinks=true]{hyperref}
\usepackage{color}
\usepackage{amsmath,amsfonts}
\usepackage{calrsfs}
\usepackage{amsmath,environ}
\DeclareMathAlphabet{\pazocal}{OMS}{zplm}{m}{n}

\newcommand{\lamba}{\lambda}
\numberwithin{equation}{section}
\usepackage{amsmath, amsthm}
    \usepackage{dsfont}
    \usepackage{fancybox}
    \usepackage{amssymb}

\newcommand{\R}{\mathbb{R}}

\newcommand{\p}{{\partial}}

\newcommand{\w}{\omega}

\newcommand{\matrice}[1]{\left( \begin{matrix}
#1
\end{matrix} \right)}

\newcommand{\epsi}{\varepsilon}
\newcommand{\sgn}{{\text{sgn}}}
\newcommand{\trace}{{\operatorname{Tr}}} 

\newcommand{\te}{\theta}
\newcommand{\lr}[1]{\langle #1 \rangle}
\newcommand{\blr}[1]{\left\langle #1 \right\rangle}

\newcommand{\HS}{{\operatorname{HS}}}
\newcommand{\tV}{{\tilde{V}}}

\newcommand{\supp}{\mathrm{supp}}

\newcommand{\Z}{\mathbb{Z}}
\newcommand{\Dd}{\mathbb{D}}

\newcommand{\Ss}{\mathbb{S}}
\newcommand{\Id}{{\operatorname{Id}}}
\newcommand{\BB}{\mathcal{B}}
\newcommand{\VV}{{\mathcal{V}}}

\newcommand{\LL}{{\mathcal{L}}}
\newcommand{\HH}{\mathcal{H}}

\newcommand{\eff}{{\operatorname{eff}}}

\newcommand{\C}{\mathbb{C}}

\newcommand{\og}{\overline{g}}

\newcommand{\az}{\alpha}

\newcommand{\tu}{{\tilde{u}}}
\newcommand{\tw}{{\tilde{v}}}

\renewcommand{\Re}{\operatorname{Re}}
\renewcommand{\Im}{\operatorname{Im}}

\newcommand{\Ee}{\mathbb{E}}

\newcommand{\Pp}{{\mathbb{P}}}

\newcommand{\hq}{{\hat{q}}}

\newcommand{\Res}{\operatorname{Res}}

\newcommand{\tsigma}{\tilde{\sigma}}

\newcommand{\1}{\mathds{1}}

\newcommand{\de}{\mathrel{\stackrel{\makebox[0pt]{\mbox{\normalfont\tiny def}}}{=}}}
\newcommand{\ras}{\mathrel{\stackrel{\makebox[0pt]{\mbox{\normalfont\tiny $\Pp$-a.s.}}}{\longrightarrow}}}
\newcommand{\law}{\mathrel{\stackrel{\makebox[0pt]{\mbox{\normalfont\tiny d}}}{\longrightarrow}}}
\newcommand{\NN}{\mathcal{N}}

\title[Resonances for random highly oscillatory potentials]{Resonances for random highly oscillatory potentials}
\date{\today}

\author{Alexis Drouot}
\email{alexis.drouot@gmail.com}

\NewEnviron{equations}{%
\begin{equation}\begin{gathered}
  \BODY
\end{gathered}\end{equation}
}

\NewEnviron{equations*}{%
\begin{equation*}\begin{gathered}
  \BODY
\end{gathered}\end{equation*}
}

\newtheorem{thm}{Theorem}
\newtheorem{cor}{Corollary}[section]
\newtheorem{rmk}{Remark}[section]
\newtheorem{lem}{Lemma}[section]

\newtheorem{theorem}[thm]{Theorem}
\theoremstyle{definition}

\begin{document}
\maketitle

\begin{abstract} We study discrete spectral quantities associated to Schr\"odinger operators of the form $-\Delta_{\R^d}+V_N$, $d$ odd. The potential $V_N$  models a highly disordered crystal; it varies randomly at scale $N^{-1} \ll 1$. We use perturbation analysis to obtain almost sure convergence of the eigenvalues and scattering resonances of $-\Delta_{\R^d} +V_N$ as $N \rightarrow \infty$. We identify a stochastic and a deterministic regime for the speed of convergence. The type of regime depends whether  the low frequencies effects due to large deviations overcome the (deterministic) constructive interference between highly oscillatory terms.\end{abstract}

\section{Introduction}\label{sec:1}

Predicting the behavior of waves scattered by a highly disordered material poses difficult practical issues. The material is usually difficult to know accurately and its defects and impurities can have a large impact on the diffusion. This motivates a general study of propagation of waves through random medias. This vast subject of research has various applications. We refer to the seminal paper of Anderson \cite{And} for the absence of diffusion of waves by certain models of condensed matter physics; to Mysak \cite{My} and Devillard--Dunlop--Souillard \cite{DDS} for water waves prediction; to the monographs of Andrew--Phillips \cite{AP} and Fehler--Maeda--Sato \cite{FSa} for applications in electromagnetism and seismography, respectively. Current mathematical research includes proofs of homogenization results and rigorous derivation of radiative transfer equations. The lecture note of Bal \cite{Bal} are a comprehensive introduction to theoretical aspects of waves in random media.

In this paper, we propose and analyze a model for waves scattered by a highly heterogeneous localized media in $\R^d$, $d$ odd. The disordered media is assumed to produce the random potential
\begin{equation}\label{eq:2s}
V_N(x) \de q_0(x) + \sum_{j \in [-N,N]^d} u_j q(Nx-j), \ \ \ \ N \gg 1, \ \ x \in \R^d,
\end{equation}
where $q_0$ and $q$ lie in $C^\infty_0(\R^d,\C)$ and $\{u_j\}_{j \in \Z^d}$ are bounded independent identically distributed random variables, with expected value $\Ee(u_j) = 0$ and variance $\Ee(u_j^2) = 1$. $V_N$ is the potential created by a localized crystal $\{j/N, \ j \in [-N,N]^d\}$ plunged in the external field $q_0$, with sites $j/N$ each generating a potential $u_j q(Nx-j)$. The scale of heterogeneity of such crystals is $N^{-1} \ll 1$.

When the $u_j$'s in \eqref{eq:2s} are replaced by $(-1)^{j_1+...j_d}$, the potential $V_N$ is deterministic and highly oscillatory. It can be seen as an idealized version of the random case, because the oscillations are perfectly alternated. In this (deterministic) idealized case, the works of Borisov--Gadyl'Shin \cite{BG}, Borisov \cite{B}, Duch\^ene--Weinstein \cite{DW}, Duch\^ene--Vuki\'cevi\'c--Weinstein \cite{DVW} and ourselves \cite{Dr2,Dr3} investigate the behavior of scattering resonances of $-\Delta_{\R^d}+V_N$ for large $N$. The present work aims to provide an extension of these works to the random case. For a pictorial comparison of the stochastic and deterministic version of $V_N$, see Figure \ref{fig:1}.

\begin{center}
\begin{figure}
{\includegraphics[width=16.5cm]{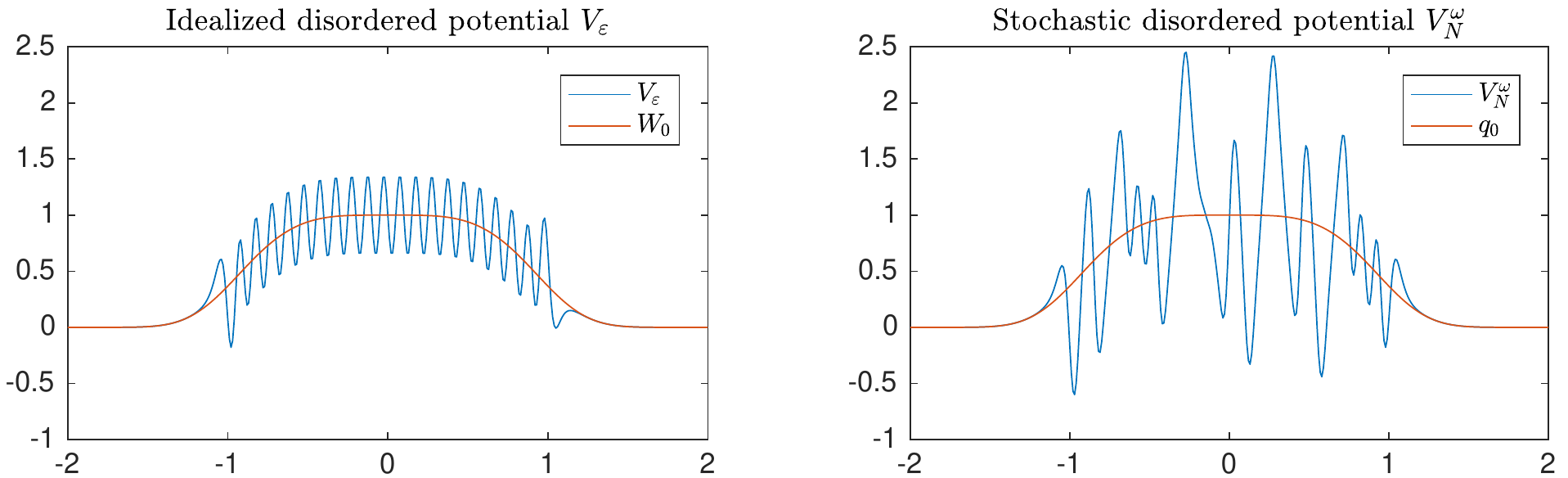}}
{\caption{On the left, deterministic version of $V_N$ studied in \cite{BG,B,DW,DVW,Dr2,Dr3}. On the right, stochastic potential $V_N$. Here $N = \epsi^{-1} = 20$ and $W_0=q_0$.}\label{fig:1}}
\end{figure}
\end{center}

\vspace*{-8mm}

We focus on scattering resonances of $V_N$, which are generalizations of eigenvalues of $-\Delta+V_N$. Generally speaking, the set of resonances $\Res(\VV)$ of a potential $\VV \in C_0^\infty(\R^d,\C)$, $d$ odd, is the set of poles of the meromorphic continuation to $\C$ of
\begin{equation*}
R_\VV(\lambda) = (-\Delta_{\R^d}+\VV- \lambda^2)^{-1} \ :\  C_0^\infty(\R^d,\C) \rightarrow C^\infty(\R^d,\C).
\end{equation*}
When $\VV$ is real-valued, the $L^2$-spectrum of $-\Delta_{\R^d}+\VV$ is purely continuous, equal to $[0,\infty)$, modulo a finite set of eigenvalues that are all negative. In this case, resonances of $\VV$ in the upper half plane are in one-to-one correspondence with eigenvalues of $-\Delta_{\R^d}+\VV$: $\lambda \in \Res(\VV)$ with $\Im \lambda > 0$ if and only if $\lambda^2$ is an eigenvalue of $-\Delta_{\R^d}+\VV$. Other resonances play the role of generalized eigenvalues for open systems. In particular, they strikingly quantize the exponential decay rates of waves scattered by $\VV$: if $u$ is a sufficiently nice solution of $(\p_t^2-\Delta_{\R^d} + V_N) u = 0$, then $u$ admits the formal expansion
\begin{equation}\label{eq:1h}
u(t,x) \sim \sum_{\lambda \in \Res(\VV)} u_{\lambda}(x) e^{i\lambda t}, \ \ u_\lambda : \R^d \rightarrow \C.
\end{equation}
(we put aside the issue of multiplicity). For every $A > 0$, $\Res(\VV) \cap \{\lambda \ : \ \Im \lambda \geq -A\}$ is a finite set; thus \eqref{eq:1h} admits a rigorous formulation in terms of exponential decay of the local energy, see e.g. \cite[Theorem 3.9]{DZ}. A comprehensive introduction to scattering resonances is found in \cite[Chapters~2~and~3]{DZ}.

\subsection{Results} We localize precisely the eigenvalues and scattering resonances of the random Schr\"odinger operator $-\Delta_{\R^d}+V_N$, where $V_N$ is the chaotic potential given in \eqref{eq:2s}. Our analysis lies within the effective media theory, which investigates whether rapidly varying terms can be replaced by suitable slowly varying terms. When $V_N$ is real-valued, our localization results transfer directly to qualitative information on the long-time behavior of waves, see e.g. the remark below Theorem \ref{thm:0}.

In the rest of the paper, $q_0$ and $q$ are two smooth compactly supported functions and 
\begin{equation*}
V_N(x) \de q_0(x) + \sum_{j \in [-N,N]^d} u_j q(Nx-j), \ \ \ \ N \gg 1, \ \ x \in \R^d, \ \ d \text{ odd.}
\end{equation*}
The potential $V_N$ is bounded and has compact support, uniformly in $N$ and of the value of the $u_j$'s, see \eqref{eq:4d} and \eqref{eq:4e} below. 
Let $\hq$ be the Fourier transform of $q$:
\begin{equation*}
\hq(\xi) \de \int_{\R^d} e^{-i\xi x} q(x) dx.
\end{equation*}
The influence of low frequencies of $q$ is well described by the order of vanishing $m$ of $\hq(\xi)$ at $\xi=0$ (i.e., the largest integer such that $\hq(\xi) = O(|\xi|^m)$ near $0$). With this notation, we define
\begin{equation*}
\gamma \de \min(7/4, d/2+m).
\end{equation*}
We recall that $\Res(\VV) \subset \C$ is the set of resonances of $\VV$ and that we denote by $m_\lambda$ the geometric multiplicity of a resonance $\lambda$, i.e. the integer
\begin{equation*}
m_\lambda \de \text{Rank } \dfrac{1}{2\pi i}\oint_\lambda \left(-\Delta_{\R^d}+\VV-\mu^2\right)^{-1} 2\mu d\mu.
\end{equation*}
The set $\Res(\VV)$ is discrete; and for any $A > 0$, there exists $B$ depending only on $A$, on the diameter of the support of $\VV$ and on $|\VV|_\infty$ such that
\begin{equation}\label{eq:4c}
\# \Res(\VV) \cap \{ \lambda: \Im \lambda \geq -A, \ |\Re(\lambda)| \geq B \} < \infty,
\end{equation}
see for instance the proof of \cite[Theorem 2.8]{DZ}.

\begin{theorem}\label{thm:1} For any $R > 0$ such that $q_0$ has no resonance on $\p \Dd(0,R)$, there exist $C, c > 0$ such that with probability $1-Ce^{-cN^\gamma}$,
\begin{equation}\label{eq:2q}
\Res(V_N) \cap \Dd(0,R) \subset \bigcup_{\lambda \in \Res(q_0) \cap \Dd(0,R)} \Dd\left(\lambda, N^{-\frac{\gamma}{2m_\lambda}}\right).
\end{equation}
Conversely, if $\lambda \in \Res(q_0) \cap\Dd(0,R)$ has multiplicity $m_\lambda$, then with probability $1-Ce^{-cN^\gamma}$, $V_N$ has exactly $m_\lambda$ resonances in  $\Dd\left(\lambda, N^{-\frac{\gamma}{2m_\lambda}}\right)$  -- counted with multiplicity. 
\end{theorem}

An application of this theorem concerns local exponential decay for waves scattered by $V_N$. Assume that $q_0$ and $q$ are real-valued and that $\Res(q_0)$ is contained in $\{ \Im \lambda < -A\}$ for some $A > 0$ (this is satisfied for instance if $q_0 \geq 0$ and $q_0 \not\equiv 0$). Let $R > 0$ such that for any $N$,
\begin{equation*}
\Res(V_N) \cap \{\Im \lambda \geq -A \} \subset \Dd(0,R).
\end{equation*} 
The bound \eqref{eq:4c} guarantees that $R$ exists. Theorem \ref{thm:1} shows that with probability $1-O(e^{-cN^\gamma})$, resonances of $V_N$ are very close to resonances of $q_0$ in $\Dd(0,R)$, in particular that $\Res(V_N) \cap \Dd(0,R) \subset \{\Im \lambda < -A\}$. The characterization of resonances as quantized exponential decay of waves \cite[Theorem 3.9]{DZ} shows that with probability $1-O(e^{-cN^\gamma})$, any solution $u : \R \times \R^d \rightarrow \C$ of 
\begin{equation}\label{eq:1c}
(\p_t^2 - \Delta_{\R^d})u=0, \ \  u(0, \cdot) \in C_0^\infty(\R^d,\C),\ \ \p_t u(0,\cdot) \in C_0^\infty(\R^d,\C), \ \ d \geq 3
\end{equation}
decays faster than $e^{-At}$:
\begin{equation*}
\forall M > 0, \ \ \sup_{|x| \leq M} |u(x,t)| = o(e^{-At}).
\end{equation*}

A combination of the Borel--Cantelli lemma with Theorem \ref{thm:1} implies the following almost-sure, non-quantitative statement:

\begin{cor}\label{cor:1} The set of accumulation points of $\Re(V_N)$ is $\Pp$-a.s. equal to $\Res(q_0)$.\end{cor}

Since the vanishing potential has a single a resonance in dimension one and none in higher dimension, Theorem \ref{thm:1} shows that all the resonances of $V_N$ with $q_0=0$ must escape to infinity as $N \rightarrow \infty$ (except one converging to $0$ in dimension one). The next result gives a lower bound on the rate of escape: 

\begin{theorem}\label{thm:0} Assume that $q_0 = 0$ and that $V_N$ is given by \eqref{eq:2s}. There exist $C, c, A > 0$ such that with probability $1-Ce^{-cN^\gamma}$, $V_N$ have no resonance above the line $\Im \lambda = -A\ln(N)$ -- apart from a single resonance in $\Dd(0,N^{-\gamma/2})$ when $d=1$.\end{theorem}

By the same argument as in the remark below Theorem \ref{thm:1},  if $d \geq 3$, $q_0 \equiv 0$ and $q$ is real-valued then solutions $u : \R \times \R^d \rightarrow \C$ of \eqref{eq:1c} must locally decay like $N^{-At}$, with probability at least $1-O(e^{-cN^\gamma})$. 

We now investigate the speed of convergence of resonances of $V_N$ to resonances of $q_0$ in a generic case. The next statement requires some preparation. A resonance $\lambda_0$ of $q_0$ is said \textit{simple} if there exist two smooth complex-valued functions $f$, $g$ on $\R^d$ -- called resonant states -- such that
\begin{equation}\label{eq:3p}
\left(-\Delta_{\R^d} + q_0 - \lambda^2\right)^{-1} - i \dfrac{f \otimes g}{\lambda-\lambda_0} \ \ \text{is holomorphic near } \lambda_0.
\end{equation}
For instance, any non-zero resonance $\lambda_0$ of $q_0$ with geometric multiplicity equal to one is simple, see for instance \cite[Theorem 3.7]{DZ}.

If $\varphi \in C^\infty(\R^d,\C)$ has real-part $\varphi_1$ and imaginary part $\varphi_2$, we define $\Sigma[\varphi]$ as the $2 \times 2$ symmetric, nonnegative matrix
\begin{equation}\label{eq:1g}
\Sigma[\varphi] \de \int_{[-1,1]^d}\matrice{  \varphi_1(x)^2 & \varphi_1(x) \varphi_2(x) \\ \varphi_1(x) \varphi_2(x) & \varphi_2(x)^2} dx.
\end{equation}
If $\Sigma[\varphi]$ is non-degenerate, we say that a complex-valued sequence of random variables $Z_j$ converges in distribution to $\NN(0,\Sigma[\varphi])$ if the multivariate random variable $(\Re(Z_j),\Im(Z_j))$ converges in distribution to the multivariate normal distribution centered at $0$ with covariance matrix $\Sigma[\varphi]$. If $\Sigma[\varphi]$ is degenerate then
\begin{equation*}
\int_{\R^d} \varphi_1(x)^2 dx \int_{\R^d} \varphi_2(x)^2 dx - \left(\int_{\R^d} \varphi_1(x) \varphi_2(x) dx\right)^2 = 0.
\end{equation*}
Hence (if, say, $\varphi_2 \not\equiv 0$), there exists $\az \in \R$ such that $\varphi_1 = \az \varphi_2$. In this situation, we say that $Z_j$ converges in distribution to $\NN(0,\Sigma[\varphi])$ if the multivariate random variable $(\Re((1+i\az)^{-1}Z_j),\Im((1+i\az)^{-1}Z_j))$ converges in distribution to $\NN\left(0,\int_{[-1,1]^d} \varphi_1(x)^2 dx\right) \otimes \delta_0$. The definition for $\varphi_1 \not\equiv 0$ is analogous. 

We will distinguish the three following cases:
\begin{itemize}
\item \textbf{Case I:} $d=1$ or $3$ and $\int_{\R^d} q(x) dx \neq 0$;
\item \textbf{Case II:} $d=1$ and $\int_\R q(x) dx = 0$, $\int_\R x q(x) dx \neq 0$ and $(f \cdot g)' \not\equiv 0$ on $[-1,1]$;
\item \textbf{Case III:} all other cases.
\end{itemize}

\begin{theorem}\label{thm:2} Under the above notations, there exist $C, c > 0$ such that the following is satisfied. For every $N$, there exists $\lambda_N$ a complex-valued random variable, such that
\begin{equation*}
\Pp(\lambda_N \in \Res(V_N)) \geq 1 - Ce^{-cN^{1/4}}
\end{equation*}
and
\begin{itemize}
\item In Case I, 
\begin{equation*}
\dfrac{N^{d/2}(\lambda_N - \lambda_0)}{i\int_{\R^d} q(x) dx} \law \NN(0,\Sigma[fg]).
\end{equation*}
\item In Case II,
\begin{equation*}
\dfrac{N^{3/2}(\lambda_N - \lambda_0)}{i\int_{\R} xq(x) dx} \law \NN(0,\Sigma[(fg)']).
\end{equation*}
\item In Case III, 
\begin{equation*}
N^2(\lambda_N-\lambda_0) \ras \dfrac{i}{(2\pi)^d}\int_{\R^d} \dfrac{\hq(\xi) \hq(-\xi)}{|\xi|^2} d\xi \cdot \int_{[-1,1]^d} f(x)g(x) dx.
\end{equation*}
\end{itemize}
\end{theorem}

When $q_0$ and $q$ are real-valued, the case of a resonance $\lambda_0 \in \Res(q_0) \cap i \R$ is of special interest. It allows to study eigenvalues of $-\Delta + q_0$: $\lambda_0 \in i(0,\infty) \cap \Res(q_0)$ if and only if $\lambda_0$ is an eigenvalue of $-\Delta + q_0$, see for instance \cite[p. 31]{DZ}. It also allows to observe the emergence of eigenvalues from the edge of the continuous spectrum. This phenomena was captured first for small  perturbations of $q_0$ in a pioneering work of Simon \cite{Si}, in dimension one and two. It was observed for highly oscillatory perturbations in \cite{BG,B,DW,DVW,Dr2,Dr3}, again in dimension one. When $q_0$ is real-valued and $\lambda_0 \in i\R$, we can pick $f = \overline{g}$ in \eqref{eq:3p} and we obtain a refinement of Theorem \ref{thm:2}, which in particular implies that eigenvalues might emerge from the edge of the continuous spectrum in dimension one:

\begin{cor}\label{cor:2} Assume that $q$, $q_0$ are real-valued and that $\lambda_0 \in i\R \cap \Res(q_0)$. Then there exist $C$, $c > 0$ such that the following is satisfied. For every $N$, there exists $\lambda_N$ a random variable with values in $i\R$ such that
\begin{equation*}
\Pp(\lambda_N \in \Res(V_N)) \geq 1 - Ce^{-cN^{1/4}}
\end{equation*}
and
\begin{itemize}
\item In Case I, 
\begin{equation*}
\dfrac{N^{d/2}(\lambda_N - \lambda_0)}{i \int_{\R^d} q(x) dx} \law \NN(0,\sigma^2), \ \ \ 
\sigma^2 \de \int_{[-1,1]^d} |f(x)|^4 dx.
\end{equation*}
\item In Case II,
\begin{equation*}
\dfrac{N^{3/2}(\lambda_N - \lambda_0)}{i\int_\R xq(x) dx} \law \NN(0,\sigma^2), \ \ \ 
\sigma^2 \de \int_{[-1,1]} \left(\left(|f|^2\right)'(x)\right)^2 dx.
\end{equation*}
\item In Case III, 
\begin{equation*}
N^2(\lambda_N-\lambda_0) \ras \dfrac{i}{(2\pi)^d} \int_{\R^d} \dfrac{|\hq(\xi)|^2}{|\xi|^2} d\xi \cdot \int_{[-1,1]^d} |f(x)|^2 dx.
\end{equation*}
\end{itemize}
If in addition $\lambda_0^2$ is an eigenvalue of $-\Delta+q_0$, then $\lambda_N^2$ is an eigenvalue with probability $1-Ce^{-cN^{1/4}}$.
\end{cor}

As an example let us assume that $d=1$, $q_0 \equiv 0$ and $q$ real-valued. The potential $q_0$ has a single resonance $\lambda_0 = 0$, which is simple and has associated resonant/coresonant states $f = \overline{g} = 1/\sqrt{2}$, see \cite[(2.2.1)]{DZ}. Theorem \ref{thm:2} shows that $V_N$ is likely to have a resonance $\lambda_N$ near $0$, which in addition belongs to $i \R$. If $\int_\R q(x) dx \neq 0$, we fall in Case I: $\lambda_N$ is roughly at distance of order $N^{-1/2}$ from $0$; precisely,
\begin{equation*}
\dfrac{2N^{1/2} \lambda_N}{i \int_\R q(x) dx}  \law \NN(0,1).
\end{equation*}
We observe $\Im \lambda_N > 0$ with probability asymptotically equal to $1/2$: with probability roughly $1/2$, $\lambda_N^2$ is an eigenvalue that emerges from the edge of the continuous spectrum of $-\Delta_\R$. If in contrast $\int_\R q(x) dx = 0$, we fall in Case III and $\lambda_N$ is at distance of order $N^{-2}$ from $0$. Precisely, if $Q$ is the compactly supported antiderivative of $q$, then
\begin{equation*}
N^2 \lambda_N \ras \dfrac{i}{2}\int_{\R} Q(x)^2 dx.
\end{equation*}
In particular, $V_N$ is likely to have a unique eigenvalue $\lambda_N^2 \sim -\frac{1}{4N^4} \int_\R Q(x)^2 dx$, which emerges from the edge of the continuous spectrum of $-\Delta_\R$. This extends results previously known in the context of highly oscillatory potentials to the random case.

\subsection{Interpretation and comments} Theorem \ref{thm:1} involve the exponent $\gamma$, which depends on $d$ and $m$. This dependence comes from large deviations: when the sequence $\{u_j\}_{j \in \Z^d}$ takes unlikely values, the potential $V_N$ differs significantly from a purely oscillatory one. This slows down the speed of convergence of resonances of $V_N$ to resonances of $q_0$. Since the number of sites ($N^d$) grows with the dimension, large deviations are less likely in higher dimensional crystals. Their effect is reduced when $m$ is large (that is, $q$ contains few low frequencies), because in such cases $q$ is inherently oscillatory, independently of the values of the random sequence $\{u_j\}_{j \in \Z^d}$. This explains the dependence of $\gamma$ on $d$ and $m$. In \S\ref{app:2}, we show on an example that Theorem \ref{thm:0} does not hold if one does not remove an event of exponentially small probability.

The works of Borisov and Gadyl'shin \cite{BG,B}, Duch\^ene--Weinstein--Vuki\'cevi\'c \cite{DVW} and ourselves \cite{Dr2,Dr3} show that the typical distance between resonances of deterministic highly oscillatory potentials and of their weak limit is of order $N^{-2}$. This is due to constructive interference between oscillatory terms. This effect is still present here. However it does not always dominate: it is sometimes overcome by large deviations, see the three cases in Theorem \ref{thm:2}. Large deviations imply stochastic corrections that are of generally of order $N^{-d/2-m}$, where $m$ is the order of vanishing of $\hq(\xi)$ at $\xi=0$. This explains why the difference between resonances of $V_N$ and of $q_0$ becomes deterministic when $d/2+m$ gets larger than $2$. This also explains why the speed of convergence of resonances of $V_N$ is not larger than $N^2$, even when $d$ or $m$ are large.

In Case III of Theorem \ref{thm:2}, the analogy with \cite{DVW,Dr2} is at its strongest and we can derive an effective potential:
\begin{equation*}
V_\eff(x) \de q_0(x)+\dfrac{i}{(2\pi)^d N^2}\int_{\R^d} \dfrac{\hq(\xi) \hq(-\xi)}{|\xi|^2} d\xi \cdot \1_{[-1,1]^d}(x).
\end{equation*}
$V_\eff$ is a small perturbation of $q_0$ whose scattering properties are very close to those of $V_N$. Specifically, near any simple resonance $\lambda_0$ of $q_0$, $V_N^\eff$ has a unique resonance $\lambda_N^\eff$. Moreover, this resonance satisfies
\begin{equation*}
\dfrac{\lambda_N-\lambda_N^\eff}{\lambda_N-\lambda_0} \ras 0
\end{equation*}
-- see for instance \cite[Lemma 3.3]{Dr2}.
An effective potential in the above sense does not exist if $d/2+m < 2$. Indeed, Theorem \ref{thm:2} implies that such a potential would be given by a distribution 
\begin{equation*}
\dfrac{1}{N^d} \int_{\R^d} q(y) dy \cdot \sum_{j \in [-N,N]^d} u_j \delta_{j/N},
\end{equation*}
which does not belong to $L^\infty$. 

Theorems \ref{thm:1} shows that resonances of $V_N$ in compact sets are very close to resonances of $q_0$ for $N$ large (with probability exponentially close to $1$). The potential $V_\# = V_N-q_0$ is in general not pointwise small, but it is small when measured with respect to a weaker topology -- see Lemma \ref{lem:1d}. This allows to treat it as a singular perturbation of $q_0$, in an abstract framework due to Golowich--Weinstein \cite{GW}.

\subsection{Relation to existing work}

To the best of our knowledge, this is the first treatment of eigenvalues and resonances for random highly oscillatory Schr\"odinger operators. The closest work is possibly Klopp \cite{K}, who derived a semiclassical Weyl law for large one-dimensional discrete ergodic systems. The potentials considered there can be seen as a high amplitude version of the potentials considered here; specifically, after rescaling, Klopp's (discrete) potential takes the form
\begin{equation*}
N^2 \sum_j u_j q(Nx-j), \ x \in \Z,\ \ \Ee(u_j) = 0, \ \Ee(u_j^2) = 1, \ \ q \in C_0(\Z, \R).
\end{equation*}

For one-dimensional deterministic highly oscillatory potentials (HOPs), Borisov and Gadyl'shin \cite{BG,B} gave necessary and sufficient conditions for the existence of a bound state, emerging from the edge of the continuous spectrum. Duch\^ene--Vuki\'cevik--Weinstein \cite{DVW} derived an explicit formula for a small effective potential, created by the constructive interference of oscillatory terms. They also obtained precise asymptotic for the transmission coefficient. We developed new techniques in \cite{Dr2, Dr3} to extend their results to higher dimensions. We obtained a full expansion for eigenvalues and resonances of HOPs, and a refined formula for the effective potential, and logarithmic resonance-free regions when the weak limit vanishes. 

On a somewhat unrelated note, Duch\^ene--Raymond \cite{DR} obtained homogenization results for large HOPs in dimension one. Dimassi \cite{Di} and Dimassi--Duong \cite{DD} used the effective Hamiltonian method of G\'erard--Martinez--Sj\"ostrand \cite{GMS} to count resonances and eigenvalues of semiclassical rescaled HOPs in any dimension $d$. They obtained a nice Weyl law in the semiclassical limit, related to papers of Klopp \cite{K2,K} and Phong \cite{P1,P2}.

Our results are a form of of stochastic stability of scattering resonances; this reinforces the possibility of observing them in physical situations. A different singular perturbation of $q_0$ was studied by Zworski \cite{Zw}, who obtained resonances as viscosity limits. For expanding dynamical systems Baladi--Young \cite{BY} investigated the stochastic stability of resonances; this was put later in a general abstract framework due to Keller--Liverani \cite{KL}. Dyatlov--Zworski \cite{DZ2} and ourselves \cite{Dr4} investigated the stochastic stability of resonances of Anosov flows. Barr\'e--M\'etivier \cite{BM} showed related results in the context of the Vlasov equation. Finally, Sj\"ostrand \cite{Sj1,Sj2} obtained semiclassical Weyl laws for certain multiplicative random perturbations of the Laplacian on both compact manifolds and on $\R^d$.

\subsection{Strategy of proof} The rest of the paper is organized in three sections. In \S\ref{sec:2}, we use Banach spaces first $\HH^{-s}, \ s > 0$ introduced by Golowich--Weinstein \cite{GW}. In these spaces, highly oscillatory elements have small norms. We continue the Golowich--Weinstein perturbation theory with respect to $\HH^{-1}$ and $\HH^{-2}$; the main tool in our approach is analytic Fredholm theory rather than the implicit function theorem. In particular, we derive in Lemma \ref{lem:1c} an exact local characteristic equation for resonances.

In \S\ref{sec:3} we use the Hanson--Wright inequality to show that with highly probability, $V_N$ can be regarded as a small $\HH^{-2}$ perturbation of $q_0$. This allows to apply the results of \S\ref{sec:2}. Ideas related to \cite{Dr2,Dr3} quickly yield Theorems \ref{thm:1} and \ref{thm:0} in \S\ref{sec:4}. Theorem \ref{thm:2} requires more attention. Lemma \ref{lem:1c} shows that with high probability, resonances $\lambda$ of $V_N$ near a simple resonance $\lambda_0$ of $q_0$ must satisfy the equation
\begin{equation}\label{eq:1e}
\lambda = \lambda_0 + a_1(V_\#, \lambda) + a_2(V_\#,\lambda) + ..., \ \ \ \ V_\# \de V_N - q_0.
\end{equation}
 In \eqref{eq:1e}, the coefficients $a_k(V_\#,\lambda)$ are holomorphic functions of $\lambda$ that depend $k$-multilinearly on $V_\#$. For $k \geq 3$, they are shown to be negligible compared to $a_1(V_\#, \lambda) + a_2(V_\#,\lambda)$. The proof of Theorem \ref{thm:2} requires a precise evaluation of $a_1(V_\#,\lambda)$ and $a_2(V_\#,\lambda)$. We apply a sophisticated version of the central limit theorem to show that $N^{d/2+m} a_1(V_\#,\lambda)$ behaves asymptotically like a Gaussian. We estimate $a_2(V_\#,\lambda)$ using Fourier analysis arguments similar to those of \cite{Dr3}. Specifically, we show that $N^2 a_2(V_\#,\lambda)$ converges almost surely to a constant, induced by constructive interference between oscillatory terms. Theorem \ref{thm:2} follows from a comparison of $a_1(V_\#,\lambda)$ and $a_2(V_\#,\lambda)$, performed in \S\ref{sec:4}. \\

\noindent \textbf{Acknowledgment.} We would like to thanks Maciej Zworski and Michael Weinstein for valuable discussions. This research was supported by the NSF grant DMS-1500852.

\section{Deterministic tools}\label{sec:2}

Let $H^s$ be the standard scale of Sobolev spaces on $\R^d$. The functional framework relevant here is a scale of Banach space $\HH^{-s}$, introduced in Golowich--Weinstein \cite{GW}. The associated norm $|\cdot|_{\HH^{-s}}$ is defined on smooth functions $\VV \in C_0^\infty(\R^d,\C)$ as the $H^s \rightarrow H^{-s}$ operator norm of the multiplication operator by $\VV$:
\begin{equation}\label{eq:3s}
|\VV|_{\HH^{-s}} \de |\VV|_{H^s \rightarrow H^{-s}} = \sup_{f \in H^s} \dfrac{|\VV f|_{H^{-s}}}{|f|_{H^s}} = \left|\lr{D}^{-s} \VV \lr{D}^{-s}\right|_\BB
\end{equation}
-- where $|\cdot|_\BB$ is the operator norm for linear operators on $L^2$. The norm on the spaces $\HH^s$ are difficult to compute explicitly. The next lemma is a bilinear upper bound for norms of functions in $\HH^{-s}$.  It supports the idea that rapidly oscillating functions have small $\HH^{-s}$ norms.

\begin{lem}\label{lem:1f} Fix $s > 0$.
\begin{enumerate}
\item[$(i)$] If $s > d/2$, then there exists $C > 0$ such that for any $\VV \in C_0^\infty$, $|\VV|_{\HH^{-s}} \leq C|\VV|_{H^{-s}}$.
\item[$(ii)$] If $0 < s \leq d/2$, then for any $s' > d/2$, there exists $C > 0$ such that 
\begin{equation*}
\VV \in C_0^\infty \ \Rightarrow \ |\VV|_{\HH^{-s}} \leq C |\VV|_\infty^{1-s/s'} |\VV|_{H^{-s'}}^{s/s'},
\end{equation*}
\end{enumerate}
\end{lem}

\begin{proof} If $s > d/2$, then the Sobolev space $H^s$ is an algebra. Therefore, there exists $C>0$ such that for any $\VV$, $f$ in $H^s$, $|\VV f|_{H^s} \leq C|\VV|_{H^s}|f|_{H^s}$. The corresponding dual inequality reads $|\VV f|_{H^{-s}} \leq C |\VV|_{H^{-s}} |f|_{H^s}$. Part $(i)$ follows now from the definition \eqref{eq:3s}.

Assume now that $0 < s \leq d/2$. By interpolation theory, for any $s' > d/2$, $|\VV|_{\HH^{-s}} \leq |\VV|_{\HH^0}^{1-\te} |\VV|_{\HH^{-s'}}^\te$ where $s = \te s'$. The $\HH^0$-norm of $\VV$ is controlled by $|\VV|_\infty$, and the $\HH^{-s'}$ -norm of $\VV$ is controlled by $|\VV|_{H^{-s'}}$ because of $(i)$. This implies $(ii)$.
\end{proof}

For $\VV \in C_0^\infty(\R^d,\C)$ and $\rho \in C_0^\infty(\R^d,\R)$ equal to $1$ on $\supp(\VV)$, we define $K_\VV(\lambda) = \VV R_0(\lambda) \rho$. Lemma \ref{lem:1i}, \ref{lem:1b} and \ref{lem:1c} below, $\VV, \VV_0, \VV_1$ denote three smooth compactly supported functions with support in $[-M,M]^d$ and bounded uniformly by $M$; and the constants $C$ depend uniformly in $M$.

\begin{lem}\label{lem:1i} Let $K$ be a compact subset of $\C$, with $0 \notin \C$ if $d=1$. There exists $C > 0$ such that for all $\lambda \in K$, 
\begin{itemize}
\item[$(i)$] The operator $K_\rho(\lambda)$ maps $L^2$ to $H^2$ and $H^{-2}$ to $L^2$, with norm controlled by $C$.
\item[$(ii)$] The operator $K_\VV(\lambda)^d$ is trace-class, with trace-class norm bounded by $C$.
\end{itemize}
\end{lem}

For the proof, see \cite[Theorem 2.1 and Lemma 3.21]{DZ}. The next results study the stability of resonances for small perturbations in $\HH^{-s}$.

\begin{lem}\label{lem:1b} Let $R > 0$ such that $\VV_0$ has no resonances in $\p \Dd(0,R)$. There exist $\epsi_0, C > 0$ such that if $C|\VV_1|_{\HH^{-2}} \leq \epsi < \epsi_0$ then
\begin{equations*}
\Res(\VV_1) \cap \Dd(0,R) \subset \bigcup_{\lambda \in \Res(\VV_0) \cap \Dd(0,R)} \Dd\left(\lambda,\epsi^{1/m_\lambda}\right), \\
\lambda \in \Res(\VV_0) \cap \Dd(0,R) \ \Rightarrow \# \Dd(\lambda_0,\epsi^{1/m_\lambda}) \cap \Res(\VV_0+\VV_1) = m_\lambda.
\end{equations*}
\end{lem}

\begin{proof} The proof is based on a Fredholm determinant approach.  To simplify the notations, we write $K_\VV$ instead of $K_\VV(\lambda)$ in this proof. We first deal with the case $d \geq 3$ and explain the modifications needed for $d=1$ at the end.

For $d \geq 3$, let $\psi$ be the entire function given by
\begin{equation*}
 \psi(z) \de (1+z)\exp\left( -z + \dfrac{z^2}{2} - ... + \dfrac{(-z)^{2d}}{2d}\right).
\end{equation*}
We define the Fredholm determinant $D_\VV(\lambda) = \det(\Id + \psi(K_\VV))$. The function $D_\VV(\lambda)$ is entire; and since $R_0(\lambda)$ has no poles, the zeroes of $D_\VV(\lambda)$ are exactly the resonances of $\VV$ in $\C$, with their multiplicity -- see \cite[Theorem 5.4]{GLMZ}. We show below that
\begin{equation}\label{eq:3d}
\sup_{\lambda \in \Dd(0,R)} \left|D_{\VV_0+\VV_1}(\lambda) - D_{\VV_0}(\lambda)\right| \leq C \left|\VV_1\right|_{\HH^{-2}}.
\end{equation}

We first observe that
\begin{equations}\label{eq:2a}
D_{\VV_0+\VV_1}(\lambda) - D_{\VV_0}(\lambda) = \int_{t=0}^1 \p_t D_{\VV_0+t\VV_1}(\lambda) dt \\
 = \int_{t=0}^1 D_{\VV_0 + t \VV_1}(\lambda) \trace\left( (\Id + \psi(K_{\VV_0+t\VV_1}))^{-1} \p_t \psi(K_{\VV_0+t\VV_1})\right) dt.
\end{equations}
Since $\psi$ is an entire function with $\psi(z) = O(z^{2d+1})$ near $z=0$, we can write $\psi(z) = \sum_{m=2d+1}^\infty a_m z^m$. The convergence is uniform convergence for $z$ in bounded subsets of $\C$. Hence,
\begin{equation}\label{eq:2b}
\p_t \psi(K_{\VV_0+t\VV_1}) = \sum_{m=2d+1}^\infty a_m \dfrac{d}{dt} K_{\VV_0+t\VV_1}^m = \sum_{m=2d+1}^\infty a_m \sum_{j+\ell = m-1} K_{\VV_0+t\VV_1}^j K_{\VV_1} K_{\VV_0+t\VV_1}^\ell.
\end{equation}
The convergence is uniform in the space of trace-class operators: when $j + \ell \geq 2d$, Lemma \ref{lem:1i} implies that  either  $K_{\VV_0+t\VV_1}^j$ or $K_{\VV_0+t\VV_1}^\ell$ is trace-class and that
\begin{equation*}
|\p_t \psi(K_{\VV_0+t\VV_1})|_\LL \leq \sum_{m=2d+1}^\infty (m-1) |a_m| \cdot  C^m
\end{equation*}
-- where $|\cdot|_\LL$ denotes the trace-class norm. This series converges absolutely because $\psi$ is an entire function of order $2d$, therefore $\{|a_m|\}_{m \geq 0}$ converges rapidly to $0$ -- see for instance \cite[(4.7)]{Dr2} for a precise statement. The cyclicity of the trace implies that  
\begin{equation*}
\trace\left( (\Id + \psi(K_{\VV_0+t\VV_1}))^{-1} K_{\VV_0+t\VV_1}^j K_{\VV_1} K_{\VV_0+t\VV_1}^\ell\right) = \trace\left( (\Id + \psi(K_{\VV_0+t\VV_1}))^{-1} K_{\VV_0+t\VV_1}^{j+\ell} K_{\VV_1}\right). 
\end{equation*}
We combine this identity with \eqref{eq:2a} and \eqref{eq:2b} to get
\begin{equations*}
\trace\left( (\Id + \psi(K_{\VV_0+t\VV_1}))^{-1} \p_t \psi(K_{\VV_0+t\VV_1}) \right) = \sum_{m=2d+1}^\infty m a_m  \trace\left( (\Id + \psi(K_{\VV_0+t\VV_1}))^{-1} K_{\VV_0+t\VV_1}^{m-1} K_{\VV_1} \right) \\
 =  \trace\left( (\Id + \psi(K_{\VV_0+t\VV_1}))^{-1} \psi'(K_{\VV_0+t\VV_1}) K_{\VV_1} \right).
\end{equations*}
Since $(1+\psi(z))^{-1} \psi'(z) = (1+z)^{-1} z^{2d}$, we obtain
\begin{equations*}
D_{\VV_0+\VV_1}(\lambda) - D_{\VV_0}(\lambda) = 
 \int_{t=0}^1 D_{\VV_0 + t \VV_1}(\lambda) \trace\left( (\Id + K_{\VV_0+t\VV_1})^{-1} K_{\VV_0+t\VV_1}^{2d} K_{\VV_1}\right) dt \\
 = \int_{t=0}^1 D_{\VV_0 + t \VV_1}(\lambda) \trace\left( (\Id + K_{\VV_0+t\VV_1})^{-1} K_{\VV_0+t\VV_1}^{2d-1} K_{\VV_1} K_{\VV_0+t\VV_1}\right) dt.
\end{equations*}
Hence, $|D_{\VV_0+\VV_1}(\lambda) - D_{\VV_0}(\lambda)|$
\begin{equation}\label{eq:2t}
 \leq \sup_{t \in [0,1]} \left|D_{\VV_0 + t \VV_1}(\lambda) (\Id + K_{\VV_0+t\VV_1})^{-1}\right|_\BB \cdot |K_{\VV_0+t\VV_1}^{2d-1}|_\LL \cdot |K_{\VV_1} K_{\VV_0+t\VV_1}|_\BB.
\end{equation}
We show that the RHS of \eqref{eq:2t} is uniformly bounded for $\lambda \in K$. If $p$ is the polynomial such that $\psi(z) = (1+z) e^{p(z)}$, then
\begin{equation}\label{eq:3t}
D_{\VV_0 + t \VV_1}(\lambda) (\Id + K_{\VV_0+t\VV_1})^{-1} = D_{\VV_0 + t \VV_1}(\lambda) (\Id + \psi(K_{\VV_0+t\VV_1}))^{-1} \cdot e^{p(K_{\VV_0+t\VV_1})}.
\end{equation}
The first factor in the RHS of \eqref{eq:3t} is controlled by \cite[Appendix 5.1]{Dr2} while the second factor is uniformly bounded by Lemma \ref{lem:1i}. The term $|K_{\VV_0+t\VV_1}^{2d-1}|_\LL$ in the RHS of \eqref{eq:2t} is also uniformly bounded because of Lemma \ref{lem:1i}; finally, $|K_{\VV_1} K_{\VV_0+t\VV_1}|_\BB$ is controlled as follows:
\begin{equation*}
|K_{\VV_1} K_{\VV_0+t\VV_1}|_\BB \leq |K_\rho \lr{D}^2|_\BB \cdot |\lr{D}^{-2} \VV_1 \lr{D}^{-2}|_\BB \cdot |\lr{D}^2 K_{\VV_0+t\VV_1}|_\BB \leq C |\lr{D}^{-2} \VV_1 \lr{D}^{-2}|_\BB,
\end{equation*}
where the boundedness of $K_\rho \lr{D}^2$ and $\lr{D}^2 K_\rho$ follow from Lemma \ref{lem:1i}. This shows \eqref{eq:3d}.

The Fredholm determinant $D_{\VV_0}(\lambda)$ has no zeroes on $\p \Dd(0,R)$, hence there exists $t > 0$ such that
$|D_{\VV_0}(\lambda)| > t$ for $\lambda \in \p \Dd(0,R)$. Hence, if $C|\VV_1|_{\HH^{-2}} \leq \epsi$,
\begin{equation}\label{eq:3e}
\lambda \in \Dd(0,R) \ \Rightarrow |D_{\VV_0+\VV_1}(\lambda)-D_{\VV_0}(\lambda)| \leq \epsi.
\end{equation}
If $\epsi$ is sufficiently small, the RHS is bounded by $t$. Rouch\'e's theorem implies that
\begin{equation}\label{eq:2p}
C|\VV_1|_{\HH^{-2}} \leq \epsi \ \Rightarrow 
\# \Res(\VV_1 + \VV_0) \cap \Dd(0,R) = \# \Res(\VV_0) \cap \Dd(0,R).
\end{equation}
Let $\lambda_0 \in \Dd(0,R)$ be a resonance of $\VV_0$ with geometric multiplicity $m_{\lambda_0}$. We show that $\VV_0+\VV_1$ has exactly $m$ resonances in the disk $\Dd(\lambda_0, \epsi^{1/m_{\lambda_0}})$ for $C$ sufficiently large, and $\epsi$ sufficiently small. There exists $r_0 > 0$ such that $\lambda_0$ is the only zero of $D_{\VV_0}(\lambda)$ on $\Dd(\lambda_0,r_0)$. Hence, $|D_{\VV_0}(\lambda)| > c_0|\lambda-\lambda_0|^{m_{\lambda_0}}$ for $c_0$ sufficiently small and $\lambda \in \Dd(\lambda_0,r_0)$. Because of this and \eqref{eq:3e}, after possibly increasing the value of $C$,
\begin{equation*}
\lambda \in \p \Dd(0,\epsi^{1/m_{\lambda_0}}) \ \Rightarrow \ |D_{\VV_0+\VV_1}(\lambda)-D_{\VV_0}(\lambda)| < |D_{\VV_0}(\lambda)|.
\end{equation*}
Again, Rouch\'e's theorem implies that $\VV_0+\VV_1$ and $\VV_0$ have the same number of zeroes in $\Dd(\lambda_0,\epsi^{1/m_{\lambda_0}})$ -- i.e. $m_{\lambda_0}$. This fact, combined with \eqref{eq:2p}, implies that all resonances of $\VV_0+\VV_1$ in $\Dd(0,R)$ are confined in 
\begin{equation*}
\bigcup_{\lambda_0 \in \Res(\VV_0)} \Dd(\lambda_0,\epsi^{1/m_{\lambda_0}}).
\end{equation*}
This concludes the proof of the lemma for $d \geq 3$.

When $d=1$, the estimate \eqref{eq:3e} holds uniformly locally on $\Dd(0,R) \setminus 0$; however, unless $\int_\R \VV(x) dx = 0$, the function $D_\VV(\lambda)$ has an essential singularity at $\lambda=0$. We introduce
\begin{equation*}
d_\VV(\lambda) \de \lambda \det(\Id + K_\VV(\lambda)),
\end{equation*}
which is an entire function of $\lambda$, and whose zeroes are exactly the resonances of $\VV$ counted with multiplicity -- see \cite[Theorem 2.6]{DZ}. Since $\trace(K_\VV) = \frac{i}{2\lambda} \int_\R \VV$ (see \cite[(2.2.1)]{DZ}) and $\psi(z) = (1+z)e^{-z}$,
\begin{equation*}
D_\VV(\lambda) = \det((\Id + K_\VV)e^{-K_\VV}) = \det(\Id + K_\VV) \cdot e^{-\trace(K_\VV)} = \lambda^{-1} d_\VV(\lambda) \exp\left( - \dfrac{i}{2\lambda} \int_\R \VV \right).
\end{equation*}
It follows that
\begin{equation*}
d_\VV(\lambda)  = \lambda D_\VV(\lambda)\exp\left( \dfrac{i}{2\lambda} \int_\R \VV \right).
\end{equation*}
Hence, to deal with $d=1$, it suffices to replace $D_\VV$ by $d_\VV$ and essentially show
\begin{equation}\label{eq:3f}
\lambda \in \Dd(0,R) \ \Rightarrow |d_{\VV_0+\VV_1}(\lambda)-d_{\VV_0}(\lambda)| \leq C \epsi.
\end{equation}
By the maximum principle, \eqref{eq:3f} holds on $\Dd(0,R)$ if it holds on $\p \Dd(0,R)$. The estimate \eqref{eq:3e} works when $d=1$ and $\lambda$ is away from $0$ -- for instance $\lambda \in \p \Dd(0,R)$. Hence, \eqref{eq:3f} holds if we can show 
\begin{equation*}
\left|\int_\R \VV_0+\VV_1 - \int_\R \VV_0 \right| = O(\epsi).
\end{equation*}
This follows from the condition $|\VV_1|_{\HH^{-2}} \leq \epsi$: if $\rho \in C_0^\infty$ is $1$ on $\supp(\VV)$, 
\begin{equation*}
\left|\int_\R \VV_0+\VV_1 - \int \VV_0 \right| = \left| \int_\R \VV_1 \rho^2 \right| \leq |\VV_1 \rho|_{H^{-2}} |\rho|_{H^2} \leq |\VV_1|_{\HH^{-2}} |\rho|_{H^2}^2 = O(\epsi).
\end{equation*}
This concludes the proof for $d=1$.
\end{proof}

Fix $\lambda_0 \in \C$ and assume that $\lambda_0 \in \Res(\VV_0)$ is simple: there exist $f, g \in C^\infty(\R^d)$ and a holomorphic family of operators $L^{\lambda_0}_{\VV_0}(\lambda)$ near $\lambda_0$ such that
\begin{equation}\label{eq:1q}
R_{\VV_0}(\lambda) = L^{\lambda_0}_{\VV_0}(\lambda) + i\dfrac{f \otimes g}{\lambda-\lambda_0}. 
\end{equation}
Under this assumption on $\lambda_0$, and assuming that $\VV_1$ is small in $\HH^{-2}$, we can write a local characteristic equation for resonances.

\begin{lem}\label{lem:1c} Under the notations of \eqref{eq:1q}, there exist $r_0, \delta_0$ and $C$ all positive such that if $|\VV_1|_{\HH^{-2}} \leq \delta_0$, then:
\begin{itemize}
\item[$(i)$] For any $\kappa \geq 0$,
\begin{equation}\label{eq:1d}
\lambda \in \Dd(0,R) \ \Rightarrow \ \sum_{k = \kappa}^\infty \left|\blr{(\VV_1 L^{\lambda_0}_{\VV_0}(\lambda) \rho)^k \VV_1 f, g}\right| \leq C |\VV_1|^{\kappa+1}_{\HH^{-1}}.
\end{equation}
\item[$(ii)$] The potential $\VV_0+\VV_1$ has a unique resonance $\lambda_1$ in $\Dd(\lambda_0,r_0)$.
\item[$(iii)$] If $\varphi : \Dd(\lambda_0,r_0) \rightarrow \C$ is the holomorphic function given by
\begin{equation}\label{eq:3q}
\varphi(\lambda) = \lambda - \lambda_0 + i \sum_{k = 0}^\infty (-1)^k  \blr{ (\VV_1 L^{\lambda_0}_{\VV_0}(\lambda) \rho)^k \VV_1 f , \overline{g}}
\end{equation}
then $\varphi$ has a unique zero $\lambda_2$ in $\Dd(\lambda_0,r_0)$. In addition, $\lambda_2 = \lambda_1$.
\end{itemize}
\end{lem}

\begin{proof} Start with $(i)$. Let $r_0 > 0$ such that $\VV_0$ has no resonances but $\lambda_0$ in $\Dd(\lambda_0,r_0)$. Below we work with $\lambda \in \Dd(\lambda_0,r_0)$. We first check that $\rho L^{\lambda_0}_{\VV_0}(\lambda) \rho$ maps $H^{-1}$ to $H^1$. We have:
\begin{equation}\label{eq:2c}
\rho L^{\lambda_0}_{\VV_0}(\lambda) \rho = \oint_{\p \Dd(\lambda_0,r_0)} \dfrac{\rho R_{\VV_0}(\mu) \rho}{\mu-\lambda} d\mu.
\end{equation}
This comes from the Cauchy formula applied to $L^{\lambda_0}_{\VV_0}(\lambda) = R_{\VV_0}(\lambda) - i\frac{f \otimes g}{\lambda - \lambda_0}$ and the identity
\begin{equation*}
\oint_{\p \Dd(\lambda_0,r_0)} \dfrac{i f \otimes g }{(\mu-\lambda)(\mu-\lambda_0)} d\mu
= 0.
\end{equation*}
When $\mu \in \p\Dd(\lambda_0,r_0)$, the operator $\rho R_{\VV_0}(\mu) \rho$ maps $L^2$ to $H^2$, see \cite[Theorem 2.2]{DZ}. Since \eqref{eq:2c} expresses $\rho L^{\lambda_0}_{\VV_0}(\lambda) \rho$ as the integral of a smoothly varying $L^2 \rightarrow H^2$ operator-valued function over a circle, the operator $\rho L^{\lambda_0}_{\VV_0}(\lambda) \rho$  is itself bounded from $L^2$ to $H^2$. Using a duality argument, it also maps $H^{-2}$ to $H^2$, and (by interpolation) $H^{-1}$ to $H^1$. This shows that for $\lambda \in \Dd(\lambda_0,r_0)$ the operator $B(\lambda) = \lr{D} \rho L^{\lambda_0}_{\VV_0}(\lambda) \rho \lr{D}$ is bounded on $L^2$. As operators on $L^2$,
\begin{equation}\label{eq:3h}
\lr{D}^{-1}(\VV_1 L^{\lambda_0}_{\VV_0}(\lambda) \rho)^k \VV_1 \lr{D}^{-1} = \left(\lr{D}^{-1} \VV_1 \lr{D}^{-1} B(\lambda)\right)^k \lr{D}^{-1} \VV_1 \lr{D}^{-1}.
\end{equation}
Since $f$ and $g$ are locally in $H^1$, we can use \eqref{eq:3h} to obtain
\begin{equations*}
\left|\blr{(\VV_1 L^{\lambda_0}_{\VV_0}(\lambda) \rho)^k \VV_1 f,g } \right| = \left|\blr{\lr{D}^{-1}(\VV_1 L^{\lambda_0}_{\VV_0}(\lambda) \rho)^k \VV_1 \lr{D}^{-1} \cdot \lr{D}\rho f, \lr{D} \rho g } \right| \\ \leq 
 C^{k+1} |\lr{D}^{-1} \VV_1 \lr{D}^{-1}|^{k+1} |\rho f|_{H^1} |\rho g|_{H^1} \leq C^{k+1} |\VV_1|_{\HH^{-1}}^{k+1}.
\end{equations*}
If $|\VV_1|_{\HH^{-2}} \leq \delta_0^2$ with $\delta_0 \leq 1/(2C)$, then $|\VV_1|_{\HH^{-1}} \leq 1/(2C)$ and we can sum the above inequality over $k$. This yields the bound \eqref{eq:1d}, thus part $(i)$.

Part $(ii)$ is an immediate consequence of Lemma \ref{lem:1b} -- possibly after reducing the value of $\delta_0$. We now prove part $(iii)$. We first show that if $|\VV_1|_{\HH^{-1}}$ is sufficiently small, the function $\varphi$ defined by \eqref{eq:3q} has a unique zero in $\Dd(\lambda_0,r_0)$. If $\varphi_0(\lambda) = \lambda-\lambda_0$ and $\delta_0$ is sufficiently small compared to $r_0$,
\begin{equation*}
\sup_{\p \Dd(\lambda_0} |\varphi-\varphi_0| \leq C \delta_0 < r_0 = \inf_{\p \Dd(\lambda_0,r_0)} |\varphi_0|
\end{equation*}
hence Rouch\'e's theorem applies and shows that $\varphi$ and $\varphi_0$ have the same number of zeros in $\Dd(\lambda_0,r_0)$ -- i.e. exactly one, denoted by $\lambda_2$. We now investigate the relation between the resonance $\lambda_1$ of $\VV_0+\VV_1$ and the zero $\lambda_2$ of $\varphi$. 

The first step is a relative Lippman–-Schwinger formula:
\begin{equation}\label{eq:0d}
R_{\VV_0+\VV_1}(\lambda) = R_{\VV_0}(\lambda) \left( \Id + \VV_1 R_{\VV_0}(\lambda) \rho \right)^{-1} \left( \Id - \VV_1 R_{\VV_0}(\lambda) (1-\rho) \right)
\end{equation}
(the standard Lippman–-Schwinger formula is \eqref{eq:0d} with $\VV_0 \equiv 0$, see \cite[(2.2.8)]{DZ}). When $\Im \lambda \gg 1$, we can write
\begin{equation}\label{eq:0e}
\VV_1 R_{\VV_0}(\lambda) \rho = \VV_1 R_0(\lambda) (\Id + \VV_0 R_0(\lambda))^{-1} \rho.
\end{equation}
The operator $-\Delta_{\R^d}$ has absolutely continuous spectrum equal to $[0,\infty)$. Hence the operator $R_0(\lambda)$ satisfies $|R_0(\lambda)|_\BB \leq |\lambda|^{-1}$ when $\Im \lambda \geq 1$. In particular, for $\Im \lambda \geq 1$, $\Id + \VV_0 R_0(\lambda)$ is invertible by a Neumann series and the norm of its inverse is smaller than $2$. The bound on $|R_0(\lambda)|_\BB$ and \eqref{eq:0e} imply that for $\Im \lambda$ large enough, $\VV_1 R_0(\lambda) \rho$ is bounded on $L^2$ with norm smaller than $1/2$. We deduce that $\Id + \VV_1 R_{\VV_0}(\lambda) \rho$ is invertible by a Neumann series; we use this representation of the inverse to verify \eqref{eq:0d} when $\Im \lambda$ is large:
\begin{equations*}
R_{\VV_0}(\lambda) \left( \Id + \VV_1 R_{\VV_0}(\lambda) \rho \right)^{-1} \left( \Id - \VV_1 R_{\VV_0}(\lambda) (1-\rho) \right) \\
 = R_{\VV_0}(\lambda) \sum_{k=0}^\infty (- \VV_1 R_{\VV_0}(\lambda) \rho )^k \left( \Id - \VV_1 R_{\VV_0}(\lambda) (1-\rho) \right)\\
 = R_{\VV_0}(\lambda) \left(\sum_{k=0}^\infty (- \VV_1 R_{\VV_0}(\lambda) \rho )^k - \sum_{k=0}^\infty (- \VV_1 R_{\VV_0}(\lambda) \rho)^k \VV_1 R_{\VV_0}(\lambda) (1-\rho)\right) \\
 = R_{\VV_0}(\lambda) \left(\sum_{k=0}^\infty (- \VV_1 R_{\VV_0}(\lambda) \rho )^k + \sum_{k=0}^\infty (- \VV_1 R_{\VV_0}(\lambda))^{k+1} (1-\rho)  \right) \\
 = R_{\VV_0}(\lambda) \left(\Id + \sum_{k=0}^\infty (- \VV_1 R_{\VV_0}(\lambda))^{k+1} \right) = R_{\VV_0}(\lambda) (\Id + \VV_1 R_{\VV_0}(\lambda))^{-1} = R_{\VV_0+\VV_1}(\lambda).
\end{equations*}
This identity extends meromorphically for all $\lambda \in \C$. We only need to check that $(\Id + \VV_1 R_{\VV_0}(\lambda)\rho)^{-1}$ preserves the class of functions with compact support: this is immediate for $\Im \lambda \gg 1$ thanks to the Neumann series representation; and it extends to all $\lambda$ by the unique continuation principle. This implies \eqref{eq:0d}.

Now assume that $r_0$ is sufficiently small so that $\lambda_0$ is the unique resonance of $\VV_0$ on the disk $\Dd(\lambda_0,r_0)$. Thanks to \eqref{eq:0d}, resonances of $\VV_0+\VV_1$ in the punctured disk $\Dd(\lambda_0,r_0) \setminus \lambda_0$ are then the poles of
\begin{equation*}
\left( \Id + \VV_1 R_{\VV_0}(\lambda) \rho \right)^{-1} = \left( \Id + \VV_1 L^{\lambda_0}_{\VV_0}(\lambda) \rho + i\dfrac{\VV_1 f \otimes g \rho}{\lambda-\lambda_0} \right)^{-1}.
\end{equation*}
When $|\VV_1|_{\HH^{-2}}$ sufficiently small and $\lambda \in \Dd(\lambda_0,r_0) \setminus \lamba_0$, the operator $\Id + \VV_1 L^{\lambda_0}_{\VV_0}(\lambda) \rho$ is invertible by a Neumann series. Indeed, since $\rho L^{\lambda_0}_{\VV_0}(\lambda) \rho$ maps $L^2$ to $H^2$ and $H^{-2}$~to~$L^2$, 
\begin{equation*}
|(\VV_1 L^{\lambda_0}_{\VV_0}(\lambda) \rho)^2|_{\BB} \leq |\VV_1|_\infty |\rho L^{\lambda_0}_{\VV_0}(\lambda) \rho|_{H^{-2} \rightarrow L^2} |\VV_1|_{H^2 \rightarrow H^{-2}} |\rho L^{\lambda_0}_{\VV_0}(\lambda) \rho|_{L^2 \rightarrow H^2} \leq C |\VV_1|_{\HH^{-2}} < 1.
\end{equation*} 
Therefore, we can write
\begin{equations*}
\left( \Id + \VV_1 R_{\VV_0}(\lambda) \rho \right)^{-1} =  \left( \Id + i\dfrac{\left( \Id + \VV_1 L^{\lambda_0}_{\VV_0}(\lambda) \rho \right)^{-1} \VV_1 f \otimes g \rho}{\lambda-\lambda_0} \right)^{-1} \left( \Id + \VV_1 L^{\lambda_0}_{\VV_0}(\lambda) \rho \right)^{-1}.
\end{equations*}
Hence, $\lambda$ is a resonance of $\VV_0+\VV_1$ in the disk $\Dd(\lambda_0,r_0) \setminus \lamba_0$ if and only if
\begin{equation*}
\Id + i\dfrac{\left( \Id + \VV_1 L^{\lambda_0}_{\VV_0}(\lambda) \rho \right)^{-1} \VV_1 f \otimes f \rho}{\lambda-\lambda_0}  \text{ is not invertible. }
\end{equation*}
This operator is the sum of the identity with a rank one projector, hence it is not invertible if and only if
\begin{equation*}
1 + \dfrac{i}{\lambda-\lambda_0}\trace \left( \left( \Id + \VV_1 L^{\lambda_0}_{\VV_0}(\lambda) \rho \right)^{-1} \VV_1 f \otimes g \rho \right) = 0.
\end{equation*}
Using the Neumann series representation of $\left( \Id + \VV_1 L^{\lambda_0}_{\VV_0}(\lambda) \rho \right)^{-1}$, we obtain the characteristic equation
\begin{equation*}
\lambda - \lambda_0 + i \sum_{k = 0}^\infty (-1)^k  \blr{ (\VV_1 L^{\lambda_0}_{\VV_0}(\lambda) \rho)^k \VV_1 f , \overline{g}} = 0
\end{equation*}
which is exactly the equation $\varphi(\lambda)=0$ on $\Dd(\lambda_0,r_0) \setminus \lambda_0$. Thus $\lambda_1 \neq \lambda_0$ implies $\lambda_1=\lambda_2$. To conclude, we show that we cannot have $\lambda_1=\lambda_0$ and $\lambda_2 \neq \lambda_0$. Otherwise, we could reverse the above argument -- that showed that $\lambda_1$ is a zero of $\varphi$ --  to deduce that $\lambda_2$ is a resonance of $\VV_0+\VV_1$. But this is a contradiction, because according to 
$(ii)$ the unique resonance of $\VV_0+\VV_1$ on $ \Dd(\lambda_0,r_0)$ is $\lambda_1$, itself equal to $\lambda_0$. 
\end{proof}

\begin{rmk} The results of this section lie within the general theory developed by Golowich--Weinstein \cite{GW}. This abstract framework gives sufficient conditions on singular perturbations so that the scattering resonances are stable -- namely, the perturbation must be small in $\HH^{-2}$. In the context of resonances for dynamical systems, a somewhat similar framework was developed in Keller--Liverani \cite{KL}. 
\end{rmk}

\begin{rmk} \cite[Theorem 4.1]{GW} asserts that near a simple resonance of $\VV_0$, there must exist a simple resonance of $\VV_0+\VV_1$ -- that depends analytically on $\VV_1$. Their proof relies on the implicit function theorem. The use of analytic Fredholm techniques instead allows us to refine this result. Lemma \ref{lem:1b} deals with higher multiplicity and shows that (conversely) every resonance of $\VV_0+\VV_1$ must be close to a resonance of $\VV_0$. Lemma \ref{lem:1c} characterizes exactly resonances of $\VV_0+\VV_1$ near $\lambda_1$, in terms of the nodal set of a holomorphic function defined as a rapidly converging power series (see $(i)$).
\end{rmk}

\section{Large $N$ asymptotic for terms related to $V_N$}\label{sec:3}

We consider now given a sequence $\{u_j\}_{j \in \Z^d}$ of independent identically distributed random variables, with
\begin{equation*}
\Ee(u_j) = 0, \ \ \Ee(u_j^2) = 1, \ \ u_j \in L^\infty.
\end{equation*}
Unless specified otherwise, all the sums below are realized over indices in $[-N,N]^d$. We fix $q, q_0 \in C_0^\infty(\R^d,\C)$ and we define $V_N = q_0 + V_\#$, where $V_\#$ is the random potential
\begin{equation*}
V_\#(x) \de \sum_{j} u_j q(Nx-j) \ \  N \gg 1.
\end{equation*}
The potential $V_N$ has support contained in a fixed compact set. Indeed, if $E_j$ denotes the support of $q(N\cdot-j)$, then $E_j$ is contained in a ball of radius $C/N$, centered at $j/N$. It follows that
\begin{equation}\label{eq:4d}
\supp(V_N) \subset \supp(q_0) \cup \bigcup_{j \in [-N,N]^d} E_j \subset \supp(q_0) \cup [-C-1,C+1]^d.
\end{equation}
Moreover, $V_N$ is bounded almost surely independently of $N$ or of the value of $\{u_j\}$. Indeed, since $E_j$ of $q(N\cdot-j)$ is contained in a ball of radius $C/N$, centered at $j/N$, any singleton of $\R^d$ intersects with at most $C^d$ sets $E_j$. As the $u_j$ are i.i.d. and bounded almost surely, the estimate
\begin{equation}\label{eq:4e}
|V_N(x)| \leq |q_0|_\infty + |q|_\infty \sum_j |u_j| \1_{E_j} \leq |q_0|_\infty + C^d |q|_\infty |u_j|_\infty < \infty
\end{equation}
holds independently of $N$ and of $\{u_j\}$.

The results of the previous section indicate clearly the path to follow. We will show that $V_\# = V_N-q_0$ is a $\HH^{-2}$-small perturbation of $q_0$, at least with large probability. This will imply that resonances of $V_N$ and $q_0$ must remain close. Lemma \ref{lem:1c} characterizes locally resonances of $V_N$ as the zeroes of a random holomorphic function, expressed as a rapidly converging series. We will estimate the first two terms in the series, and show that every other term is negligible.

\subsection{Probabilistic tools} 

When $\az = (\az_{j\ell})$ is a matrix with complex entries, we denote by $|\az|_\HS$ its Hilbert--Schmidt norm: $|\az|_\HS^2 = \sum_{j,\ell} |\az_{j\ell}|^2$. Our first result is a reformulation of the Hanson--Wright inequality \cite{HW} -- a well-known bound that estimate quadratic forms at random vectors. We include the proof for the sake of completeness.

\begin{lem}\label{lem:1a} There exist $c, C > 0$ depending only on the distribution of the $u_j$'s such that the following holds. For any $N^d \times N^d$ matrix $\az = (\az_{j\ell})$ with complex entries, the expected value of $\sum_{j,\ell} \az_{j\ell} u_j u_\ell$ is $\trace(\az)$. In addition, for any $t > 0$  with $t^2 \geq 2|\trace(\az)|$,  
\begin{equations}\label{eq:0a}
\Pp\left(\left|\sum_{j,\ell} \az_{j\ell} u_j u_\ell \right| \geq t^2 \right) \leq C \exp \left( - \dfrac{ct^2}{|\az|_\HS} \right).
\end{equations}
\end{lem}

\begin{proof} For $\ell \neq j$, $\Ee(u_j u_\ell) = \Ee(u_j)\Ee(u_\ell) = 0$,  hence
\begin{equation*}
\Ee\left( \sum_{j,\ell} \az_{j\ell} u_j u_\ell \right) = \sum_{j,\ell} \az_{j\ell} \Ee(u_j u_\ell) = \sum_{j} \az_{jj} \Ee(u_j^2) = \trace(\az). 
\end{equation*}
This proves the statement about the expected value. To show \eqref{eq:0a}, we first observe that if $t^2 \geq 2 |\trace(\az)|$, 
\begin{equation*}
\left|\sum_{j,\ell} \az_{j\ell} u_j u_\ell \right| \geq t^2 \Rightarrow \left|\sum_{j,\ell} \az_{j\ell} u_j u_\ell - \trace(\az) \right| \geq t^2 - |\trace(\az)| \geq \dfrac{t^2}{2}.
\end{equation*}
Therefore, 
\begin{equations*}
\Pp\left(\left|\sum_{j,\ell} \az_{j\ell} u_j u_\ell \right| \geq t^2 \right) \leq \Pp\left(\left|\sum_{j,\ell} \az_{j\ell} u_j u_\ell - \trace(\az) \right| \geq \dfrac{t^2}{2} \right) \\ \leq \exp\left( - c \min\left(\dfrac{t^2}{|\az|}, \dfrac{t^4}{|\az|_\HS^2} \right) \right).
\end{equations*}
In the above we applied the Hanson--Wright inequality \cite{HW} -- for the general version needed here and a very elegant proof we refer to Rudelson--Vershynin \cite{RV}. In the above, the constant $c$ depends only on the distribution of the $u_j$'s and $|\az|$ denotes the norm of $\az$ as an operator $\C^{N^d} \rightarrow \C^{N^d}$ with $\C^{N^d}$ provided with its Euclidean norm. If in addition we assume that $t^2 \geq |\az|_\HS$, we can use $|\az| \leq |\az|_\HS$ and $\frac{t^2}{|\az|_\HS} \geq 1$ to get
\begin{equation*}
\min\left(\dfrac{t^2}{|\az|}, \dfrac{t^4}{|\az|_\HS^2} \right) \geq \min\left(\dfrac{t^2}{|\az|_\HS}, \dfrac{t^4}{|\az|_\HS^2} \right) = \dfrac{t^2}{|\az|_\HS}. 
\end{equation*}
This implies \eqref{eq:0a} in the case $t^2 \geq 2 |\trace(\az)|$ and $t^2 \geq |\az|_\HS$. We remove the assumption $t^2 \geq |\az|_\HS$ by observing that the opposite case implies $\frac{t^2}{|\az|_\HS} \leq 1$, which leads to
\begin{equation*}
\Pp\left(\left|\sum_{j,\ell} \az_{j\ell} u_j u_\ell \right| \geq t^2 \right) \leq 1 \leq e^{c} e^{-c t^2/|\az|_\HS}.
\end{equation*}
It suffices to set $C=e^c$ to complete the proof.
\end{proof}

We will need a slightly sophisticated version of the central limit theorem, based on Lindeberg--Lyapounov's result.

\begin{lem}\label{lem:1j} Let $\varphi \in C^\infty(\R^d,\C)$, not identically vanishing on $[-1,1]^d$ and $\Sigma[\varphi]$ be the $2 \times 2$ matrix defined in \eqref{eq:1g}. Then,
\begin{equation}\label{eq:1f}
\dfrac{1}{N^{d/2} \int_{\R^d} q(x) dx}\sum_{j \in [-N,N]^d} u_j \int_{\R^d} q(x) \varphi\left( \dfrac{x+j}{N} \right) dx \law \NN(0,\Sigma[\varphi])
\end{equation}
where the convergence to $\NN(0,\Sigma[\varphi])$ is defined before Theorem \ref{thm:2}.
\end{lem}

\begin{proof} Write $\varphi=\varphi_1+i\varphi_2$. We first assume that $\Sigma[\varphi]$ is not degenerate. This is equivalent to $\varphi_1$ and $\varphi_2$ linearly independent over $\R$. Let $\sigma_{j,N}^1, \sigma_{j,N}^2$ be the two real numbers defined by 
\begin{equation*}
\sigma_{j,N}^1 + i\sigma_{j,N}^2 = \dfrac{1}{\int_{\R^d} q(x) dx} \int_{\R^d} q(x) \varphi \left( \dfrac{x+j}{N} \right) dx.
\end{equation*}
To show \eqref{eq:1f}, it suffices to study the convergence in distribution to a Gaussian of 
\begin{equation*}
\dfrac{1}{N^{d/2}}\sum_{j \in [-N,N]^d} u_j (s\sigma_{j,N}^1+t\sigma_{j,N}^2), \ \ (s, t) \neq (0,0)
\end{equation*}  
because of the Cram\'er--Wold device \cite[Theorem 29.4]{Bi}. We apply the central limit theorem in its version due to Lyapounov, see \cite[Theorem 27.3]{Bi}. We remark that Lyapounov's condition: 
\begin{equation*}
\limsup_{N \rightarrow \infty} \dfrac{1}{N^{3d/2}} \sum_{j \in [-N,N]^d} (s \sigma_{j,N}^1+t\sigma_{j,N}^2)^3 = 0
\end{equation*}
is immediately satisfied because $s \sigma_{j,N}^1+t\sigma_{j,N}^2 = O(1)$. Hence, we deduce that
\begin{equations*}
\dfrac{1}{N^{d/2}} \sum_{j \in [-N,N]^d} u_j (s\sigma_{j,N}^1+t\sigma_{j,N}^2) \rightarrow \NN(0,\sigma(s,t)^2), \\
\sigma(s,t)^2 = \lim_{N \rightarrow \infty} \dfrac{1}{N^d} \sum_{j \in [-N,N]^d} (s \sigma_{j,N}^1+t\sigma_{j,N}^2)^2.
\end{equations*}
It remains to compute $\sigma(s,t)^2$ and check that it is not vanishing. Since $\varphi$ is smooth and $q$ has compact support, a Taylor expansion and a Riemann series argument shows
\begin{equation*}
\sigma_{j,N}^k = \varphi_k\left( \dfrac{j}{N} \right)+O(N^{-1}), \ \ \sigma(s,t)^2 = \int_{[-1,1]^d} (s\varphi_1+t\varphi_2)^2(x) dx.
\end{equation*}
Since $\varphi_1$ and $\varphi_2$ are linearly independent, $\sigma(s,t) \neq 0$ as long as $(s,t) \neq (0,0)$. We recognize the distribution of $sX+tY$, where $(X,Y)$ is a Gaussian vector with mean $0$ and covariance matrix $\Sigma[\varphi]$. This proves the lemma when $\Sigma[\varphi]$ is non degenerate.

We now deal with $\Sigma[\varphi]$ degenerate. The determinant of $\Sigma[\varphi]$ vanishes; this yields the case of equality in the Cauchy--Schwarz inequality. Hence, we can assume $\varphi_1 = \az \varphi_2$ for some $\az \in \R \setminus 0$ -- the case $\varphi_2=\az \varphi_1$ is treated similarly. According to our definition of convergence to $\NN(0,\Sigma[\varphi])$, it remains to study the convergence in distribution of
\begin{equations*}
\dfrac{1}{N^{d/2} (1+i\az) \int_{\R^d} q(x) dx }\sum_{j \in [-N,N]^d} u_j \int_{\R^d} q(x) \varphi\left( \dfrac{x+j}{N} \right) dx \\ = \dfrac{1}{N^{d/2} \int_{\R^d} q(x) dx}\sum_{j \in [-N,N]^d} u_j \int_{\R^d} q(x) \varphi_1\left( \dfrac{x+j}{N} \right) dx.
\end{equations*}
Again, we let $\tsigma_{j,N}^1, \tsigma_{j,N}^2$ be the real numbers such that 
\begin{equation*}
\tsigma_{j,N}^1+ i\tsigma_{j,N}^2 = \dfrac{1}{\int_{\R^d} q(x) dx}\int_{\R^d} q(x) \varphi_1\left( \dfrac{x+j}{N} \right) dx.
\end{equation*}
We now apply Lyapounov's central limit theorem to study the convergence in distribution of 
\begin{equation*}
\dfrac{1}{N^{d/2}} \sum_{j \in [-N,N]^d} u_j (s\tsigma_{j,N}^1+t\tsigma_{j,N}^2).
\end{equation*}
As in the case $\Sigma[\varphi]$ non-degenerate, we first check Lyapounov's condition -- which is obviously verified because $\tsigma_{j,N}^k=O(1)$. In fact, we even have
\begin{equation*}
\tsigma_{j,N}^1 = \varphi_1\left( \dfrac{j}{N} \right) + O(N^{-1}), \ \ \tsigma_{j,N}^2 = O(N^{-1}),
\end{equation*}
and as previously, an evaluation by a Riemann sum yields
\begin{equations*}
\dfrac{1}{N^{d/2}} \sum_{j \in [-N,N]^d} u_j (s\tsigma_{j,N}^1+t\tsigma_{j,N}^2) \rightarrow \NN(0,\tsigma(s,t)^2), \\
\tsigma(s,t)^2 = \limsup_{N \rightarrow \infty} \dfrac{1}{N^d} \sum_{j \in [-N,N]^d} (s \tsigma_{j,N}^1+t\tsigma_{j,N}^2)^2 = s^2 \int_{[-1,1]^d} \varphi_1(x)^2 dx.
\end{equations*}
If $(X,Y)$ are independent random variables with distributions $\NN\left(0,\int_{[-1,1]^d} \varphi_1(x)^2 dx\right)$ and $\delta_0$, respectively, the random variable $sX+tY$ has distribution $\NN(0,\tsigma(s,t)^2)$. Another application of the Cram\'er--Wold device concludes the proof.
\end{proof}

\subsection{Estimates on $\HH^{-s}$-norms} Recall that $m$ is the order of vanishing of $\hq(\xi)$ at $\xi = 0$ and that $\gamma = \min(7/4,d/2+m)$. In order to apply the results of \S\ref{sec:2}, we first observe that $V_\# = V_N-q_0$ is likely small in $\HH^{-2}$. 

\begin{lem}\label{lem:1d} There exist  $C_0, c_0 > 0$ such that for any $N$,
\begin{equations*}
\Pp(|V_\#|_{\HH^{-2}} \geq N^{-\gamma/2}) \leq C_0 e^{-c_0 N^{\gamma}}
\end{equations*}
\end{lem} 

\begin{proof} We start first with $d \leq 3$. In this case $2 > d/2$, therefore Lemma \ref{lem:1f} implies that $|V_\#|_{\HH^{-2}} \leq C|V_\#|_{H^{-2}}$. Hence
\begin{equation}\label{eq:5v}
\Pp(|V_\#|_{\HH^{-2}} \geq N^{-\gamma/2}) \leq \Pp(|V_\#|_{H^{-2}}^2 \geq C^{-2}N^{-\gamma}).
\end{equation}
The advantage of the $H^2$-norm over the $\HH^{-2}$-norm is its bilinear character. This allows us to apply the Hanson--Wright inequality (Lemma \ref{lem:1a}). We observe that
\begin{equations*}
|V_\#|_{H^{-2}}^2 = \int_{\R^d} \dfrac{1}{(1+|\xi|^2)^2} \left|\int_{\R^d} e^{-ix\xi} \sum_j u_j q(Nx-j) dx \right|^2 d\xi \\
 =  \sum_{j,\ell} u_j u_\ell \int_{\R^d} \dfrac{1}{(1+|\xi|^2)^2} \int_{\R^d} e^{-ix\xi} q(Nx-j) dx \int_{\R^d} e^{-ix\xi} q(Ny-j) dy d\xi
\end{equations*}
Substitutions $x \mapsto \frac{x+j}{N}$, $y \mapsto \frac{y+\ell}{N}$, $\xi \mapsto N\xi$ yield
\begin{equation}\label{eq:2d}
|V_\#|_{H^{-2}}^2 = \sum_{j,\ell} \az_{j\ell} u_j u_\ell, \ \ \az_{j\ell} \de \dfrac{1}{N^d} \int_{\R^d} \dfrac{e^{i\xi(j-\ell) }|\hq(\xi)|^2}{(1+N^2 |\xi|^2)^2} d\xi.
\end{equation}
We recall that $\Ee(u_j) = 0$ and $\Ee(u_j^2) = 1$. Hence, $\trace(\az) = N^d \az_{00}$. If $\hq$ vanishes at order $m$, we get
\begin{equations*}
\trace(\az) = \int_{\R^d} \dfrac{|\hq(\xi)|^2}{(1+N^2 |\xi|^2)^2 } d\xi = \int_{r=0}^\infty \int_{\Ss^{d-1}} \dfrac{|\hq(r \te)|^2 r^{d-1} }{(1+N^2 r^2)^2 }dr d\sigma(\te) \\
\leq C \int_{r=0}^1 \dfrac{r^{2m+d-1} dr }{(1+N^2 r^2)^2 } + O(N^{-4}) = O(N^{-2m-d}) \int_{r=0}^N \dfrac{r^{2m+d-1} dr }{(1+r^2)^2} + O(N^{-4}) = O(N^{-2\gamma}).
\end{equations*}
In particular, $\trace(\az) = O(N^{-2\gamma})$. Since $|\az_{j\ell}| \leq |\az_{00}|$, the same computation shows $|\az|_\HS^2 \leq N^{2d} |\az_{00}|^2 = O(N^{-4\gamma})$. Lemma \ref{lem:1a} and \eqref{eq:5v} imply that for $N$ large enough,
\begin{equation}\label{eq:1b}
\Pp\left(|V_\#|_{\HH^{-2}} \geq N^{-\gamma/2}\right) \leq C e^{-cN^{\gamma}}.
\end{equation}
We may remove the assumption on $N$ by increasing the value of $C$ in \eqref{eq:1b}.

We now deal with $d \geq 5$. In this case, $\gamma = 7/4$. Fix $s > d/2$, and apply Lemma \ref{lem:1f}: $|V_\#|_{\HH^{-2}} \leq C |V_\#|_{H^{-s}}^{2/s}$. Therefore,
\begin{equations}\label{eq:2f}
\Pp\left(|V_\#|_{\HH^{-2}} \geq N^{-\gamma/2}\right) \leq \Pp\left(|V_\#|_{H^{-s}}^{2/s} \geq C^{-1} N^{-\gamma/2}\right) = \Pp\left(|V_\#|_{H^{-s}}^2 \geq C^{-s} N^{-s\gamma /2}\right).
\end{equations}
We now compute $|V_\#|_{H^{-s}}^2$: as in \eqref{eq:2d},
\begin{equation*}
|V_\#|_{H^{-s}}^2 = \sum \az_{j\ell} u_j u_\ell, \ \ \az_{j\ell} \de \dfrac{1}{N^d} \int_{\R^d} \dfrac{e^{i\xi(j-\ell) }|\hq(\xi)|^2}{(1+N^2 |\xi|^2)^s } d\xi.
\end{equation*}
We observe that $\trace(\az) = N^d \az_{00}$, thus
\begin{equations*}
\trace(\az) = \int_{\R^d} \dfrac{|\hq(\xi)|^2}{(1+N^2 |\xi|^2)^s } d\xi = \int_{r=0}^\infty \int_{\Ss^{d-1}} \dfrac{|\hq(r \te)|^2 r^{d-1} dr d\sigma(\te)}{(1+N^2 r^2)^s } \\
\leq C \int_{r=0}^1 \dfrac{r^{d-1}dr}{(1+N^2r^2)^s} + O(N^{-2s}) \sim N^{-d} + N^{-2s}.
\end{equations*}
Since $s > d/2$, we obtain $\trace(\az) = O(N^{-d})$. Similarly, $|\az|_{\HS}^2 = O(N^{-2d})$. Lemma \ref{lem:1a} shows that for any $t > 0$ with $t^s \geq O(N^{-d})$,
\begin{equation}\label{eq:2e}
\Pp\left(|V_\#|_{H^{-s}}^2 \geq C^{-s} t^{s}\right) \leq C e^{-ct^s N^d}.
\end{equation}
Fix now $s = d/\gamma > d/2$. We have $t^s N^d = (t^2 N^{2\gamma})^{d/(2\gamma)}$. If we take $t=N^{-\gamma/2}$, then \eqref{eq:2e} implies for $N$ sufficiently large
\begin{equation*}
\Pp\left(|V_\#|_{H^{-s}}^2 \geq C^{-s} N^{-s\gamma/2}\right) \leq C e^{-cN^{\gamma}}.
\end{equation*}
Again, we can get rid of the assumption on $N$ by increasing the value of $C$. The conclusion follows now from \eqref{eq:2f}.\end{proof}

Because of \eqref{eq:1d}, we also need to estimate $|V_\#|_{\HH^{-1}}$:

\begin{lem}\label{lem:1l} There exist $C, c > 0$ such that with probability $1-Ce^{-cN^{1/4}}$,
\begin{equations}\label{eq:2y}
|V_\#|_{\HH^{-1}} \leq N^{-3/8} \ \ \text{ if } d=1 \text{ and } \int_\R q(x) dx \neq 0, \\
|V_\#|_{\HH^{-1}} \leq N^{-7/8} \ \ \text{ if } d \geq 3; \text{ or } d=1 \text{ and } \int_\R q(x) dx = 0.
\end{equations}
\end{lem}

\begin{proof} The proof is similar to that of Lemma \ref{lem:1d}. We start with $d=1$. In this case, $|V_\#|_{\HH^{-1}} \leq C |V_\#|_{H^{-1}}$. As in \eqref{eq:2d},
\begin{equation*}
|V_\#|_{\HH^{-1}}^2 = \sum_{j,\ell} \az_{j\ell}u_j u_\ell, \ \ \ \ \ \az_{j\ell} \de \dfrac{1}{N} \int_{\R} \dfrac{e^{i\xi(j-\ell) }|\hq(\xi)|^2}{1+N^2 |\xi|^2} d\xi.
\end{equation*}
We observe that $\trace(\az) = O(N^{-1})$. Similarly, $|\az|_\HS^2 = O(N^{-2})$. The Hanson--Wright inequality implies
\begin{equation*}
\Pp(|V_\#|_{\HH^{-1}} \geq N^{-3/8}) = O(e^{-cN^{1/4}}).
\end{equation*}
If in addition $\int_\R q(x) dx = 0$ then we can split the integral defining $\az_{00}$ in low and high frequency parts -- as in the proof of Lemma \ref{lem:1d} -- and obtain $\az_{00} = O(N^{-3})$, $\trace(\az) = O(N^{-2})$ and $|\az|_\HS = O(N^{-4})$. Hence, 
\begin{equation*}
d=1, \ \int_\R q(x) dx = 0 \ \Rightarrow \ \Pp(|V_\#|_{\HH^{-1}} \geq N^{-7/8}) = O(e^{-cN^{1/4}}).
\end{equation*} 

We continue the proof for $d \geq 3$. Because of Lemma \ref{lem:1f}, for $s > d/2$, $|V_\#|_{\HH^{-1}} \leq C |V_\#|_{H^{-s}}^{1/s}$. Apply \eqref{eq:2e} (which is also valid for any $d$) to $t=N^{-8/5}$ obtain
\begin{equation*}
\Pp(|V_\#|_{H^{-s}} \geq N^{-7/8}) = O(e^{-cN^{d-7s/4}}). 
\end{equation*} 
Since $d \geq 3$, there exists $s > d/2$ such that $d-7s/4 = 1/4$. The lemma follows.\end{proof}

\subsection{Estimates on terms in the expansion \eqref{eq:3q}}

The previous lemmas show that $V_N$ is a small perturbation of $q_0$. This allows to apply the general theory developed in \S\ref{sec:2}. At this point it is simple to obtain that resonances of $V_N$ converge to those of $q_0$, almost surely (Corollary \ref{cor:1}). We aim to precise this result. Below, $\lambda_0$ denotes a simple resonance of $q_0$: there exists $f, g \in C^\infty(\R^d)$ and $L_{q_0}^{\lambda_0}(\lambda)$ a family of operators that is holomorphic near $\lambda_0$, such that
\begin{equation*}
R_{q_0}(\lambda) = L_{q_0}^{\lambda_0}(\lambda) + i\dfrac{f \otimes g}{\lambda- \lambda_0}.
\end{equation*}
In order to refine Corollary \ref{cor:1}, we will use Lemma \ref{lem:1c}. This requires to estimate the first two terms which appear in \eqref{eq:3q}: $\lr{V_\# f,g}$ and $\lr{V_\# L_{q_0}^{\lambda_0}(\lambda) V_\# f,\og}$.

\begin{lem}\label{lem:1e} If $\int_{\R^d}q(x) dx \neq 0$, 
\begin{equations}\label{eq:1i}
\dfrac{N^{d/2}}{\int_{\R^d} q(x) dx} \lr{V_\# f, \overline{g}} \law \NN(0,\Sigma[fg]), \text{ as } N \rightarrow \infty, \\  \Pp(N^{d/2}|\lr{V_\# f,g}| \geq N^{1/4}) = O(e^{-cN^{1/2}}).
\end{equations}
\end{lem}

\begin{proof} We have:
\begin{equation*}\label{eq:2k}
\lr{V_\# f, \overline{g}} = \sum_j u_j \int_{\R^d} q(Nx-j) f(x) g(x) dx = \dfrac{1}{N^d} \sum_j u_j \int_{\R^d} q(x) (fg)\left(\dfrac{x+j}{N}\right) dx.
\end{equation*}
If $fg$ is not identically vanishing on $[-1,1]^d$, then Lemma \ref{lem:1j} applies and yields \eqref{eq:1i}. It remains to show the non-vanishing condition. If $fg$ is identically vanishing on $[-1,1]^d$, either $f$ or $g$ vanishes on an open set $\Omega \subset [-1,1]^d$. Assume that $f$ vanishes on $\Omega$ and define $E = \R^d \setminus \supp(f)$. This is an open subset of $\R^d$ containing $\Omega$ and we show that it is also closed using the unique continuation principle. Since $f$ is a resonant state, 
\begin{equation*}
-\Delta f = (-q_0+\lambda^2)f. 
\end{equation*}
This can be seen for instance by observing that $(-\Delta + q_0-\lambda^2) R_{q_0}(\lambda) = \Id$ (which is holomorphic), hence $(-\Delta + q_0-\lambda^2) f \otimes g = 0$. Therefore, for any $x_0 \in \text{adh}(E)$ and any $x \in B(x_0,1)$,
\begin{equation*}
|\Delta f(x)| = |(-q_0(x)+\lambda^2)f(x)-(-q_0(x_0)+\lambda^2)f(x_0)| \leq (|q_0|_\infty + |\lambda|^2) \sup_{B(x_0,1)} |\nabla f|.
\end{equation*}

In addition, since $f$ is smooth and vanishes at infinite order at $x_0$, $f(x) = O(|x-x_0|^N)$ for any $N$. \cite[Theorem 17.2.6]{Ho} applies and shows that $f$ vanishes on $B(x_0,1)$. Hence $E$ is closed and $f \equiv 0$, which is not possible. If now $g$ vanishes on $\Omega$, then the same argument using that $g$ is a coresonant state:
\begin{equation*}
-\Delta g = (-\overline{q_0}+\overline{\lambda}^2)g
\end{equation*}
implies $g \equiv 0$, which is not possible either. This shows the convergence in distribution of $\lr{V_\# f,\og}$.

To show the large deviation estimate, we first write
\begin{equation}\label{eq:3k}
|\lr{V_\# f, \overline{g}}|^2 = \sum_{j,\ell} u_j u_\ell \az_j \overline{\az_\ell}, \ \ \az_j \de \dfrac{1}{N^d}\int_{\R^d} q(x) (fg)\left(\dfrac{x+j}{N}\right)dx.
\end{equation}
Since $f$ and $g$ are bounded, $\az_j = O(N^{-d})$, $\sum_j |\az_j|^2  = O(N^{-d})$, $|\az_j \az_\ell|_\HS^2 = O(N^{-2d})$. We can then apply the Hanson--Wright inequality to obtain $\Pp(|\lr{V_\# f,g}| \geq N^{1/4-d/2}) = O(e^{-cN^{1/2}})$, as claimed. \end{proof}

\begin{lem}\label{lem:1k} Assume that $d=1$ and $\int_\R q(x) dx = \int_\R xq(x) dx = 0$, or that $d=3$ and $\int_{\R^3} q(x) dx = 0$. Then, 
\begin{equation*}
\Pp(|\lr{V_\# f, \overline{g}}| \geq N^{-9/4}) = O(e^{-cN^{1/2}}).
\end{equation*}
\end{lem}

\begin{proof} If $d=1$ and $\int_\R q(x) dx = \int_\R xq(x) dx = 0$, we can find $Q \in C_0^\infty(\R,\C)$ such that $Q'' = q$. We use the bilinear expression \eqref{eq:3k} for $|\lr{V_\# f,g}|^2$. Thanks to a double integration by parts, we see that $\az_j = O(N^{-3})$:
\begin{equation*}
\az_j = \dfrac{1}{N}\int_{\R} Q''(x) (fg)\left(\dfrac{x+j}{N}\right)dx = \dfrac{1}{N^3} \int_{\R} Q(x) (fg)''\left(\dfrac{x+j}{N}\right)dx.
\end{equation*}
It follows that $\sum_j |\az_j|^2 = O(N^{-5})$ and $|\az_j \az_\ell|_\HS = O(N^{-5})$. Hence, the Hanson--Wright inequality yields $\Pp(|\lr{V_\# f, \overline{g}}| \geq N^{-9/4}) \leq Ce^{-cN^{1/2}}$. 

If $d=3$ and $\int_\R q(x) dx  = 0$, we use again \eqref{eq:3k}. Since $f$ and $g$ are smooth,
\begin{equation}\label{eq:3r}
(fg)\left( \dfrac{x+j}{N} \right) = (fg)\left( \dfrac{j}{N} \right) + O(N^{-1})
\end{equation}
uniformly for $j \in [-N,N]^3$ and $x \in \supp(q)$. Using that $\int_{\R^3} q(x) dx = 0$, we deduce
\begin{equation*}
\az_j = \dfrac{1}{N^3} \int_{\R^3} q(x) (fg)\left( \dfrac{x+j}{N} \right) dx = O(N^{-4}).
\end{equation*}
In particular, $|\az_j \az_\ell|_\HS^2 = |\az_j|_{\ell^2}^4$. Using that $\int_{\R^3} q(x) dx = 0$ and \eqref{eq:3r}, we see that $\az_j = O(N^{-4})$. Therefore, $\sum_j |\az_j|^2 = O(N^{-5})$ and $|\az_j \az_\ell|_\HS^2 = O(N^{-10})$ and we conclude as above.\end{proof}

\begin{lem}\label{lem:1g} Assume that $d=1$, $\int_{\R}q(x) dx = 0$ and $\int_{\R} xq(x) dx \neq 0$. As $N \rightarrow +\infty$,
\begin{equations}\label{eq:1j}
(fg)' \not\equiv 0 \text{ on } [-1,1] \ \ \Rightarrow \ \   \dfrac{N^{3/2}}{\int_{\R} xq(x) dx} \lr{V_\# f, \og} \law \NN(0,\Sigma[(fg)']), \\
(fg)' \equiv 0 \text{ on } [-1,1] \ \ \Rightarrow \ \  \lr{V_\# f, \og} = O(N^{-3}).
\end{equations}
\end{lem}

\begin{rmk} In practice, we can have $(fg)' \equiv 0$ on $[-1,1]$: for instance if $q_0 = 0$ and $\lambda_0 = 0$, then $f$ and $g$ are constant functions -- see the discussion following Corollary \ref{cor:2}. \end{rmk}

\begin{proof} As above,
\begin{equation*}
\lr{V_\# f, \overline{g}} = \dfrac{1}{N} \sum_j u_j \int_{\R} q(x) (fg)\left(\dfrac{x+j}{N}\right) dx.
\end{equation*}
Since $\int_{\R^d}q(x)dx = 0$, there exists a unique $Q \in C_0^\infty(\R)$ such that $Q'=q$. Integrating by parts in the above yields
\begin{equation*}
 \lr{V_\# f, \overline{g}} =  -\dfrac{1}{N^2} \sum_j u_j \int_{\R} Q(x) (fg)'\left(\dfrac{x+j}{N}\right) dx.
\end{equation*}
If $(fg)'$ is not identically vanishing on $[-1,1]$, then the first implication of \eqref{eq:1j} follows from $\int_\R xq(x) dx = \int_\R Q(x) dx$ and Lemma \ref{lem:1j}.

If now $(fg)'$ vanishes on $[-1,1]$ then
\begin{equation}\label{eq:2l}
\lr{V_\# f, \overline{g}} = -\dfrac{1}{N^2} \sum_j u_j \int_{|x+j| \geq N} Q(x) (fg)'\left(\dfrac{x+j}{N}\right) dx.
\end{equation}
Let $L > 0$ such that $\supp(Q) \subset [-L,L]$. The indices $j$ such that $Q$ does not vanish identically on the set $\{|x+j| \geq N\}$ must satisfy $|x+j| \geq N$ for some $|x| \leq L$, in particular $|j| \geq N-L$. This happens for at most $2L$ values of the $j$'s, that must remain at fixed distance from $\pm N$. For such $j$'s,
\begin{equation*}
(fg)'\left(\dfrac{x+j}{N}\right) = (fg)'\left(\dfrac{x+j \mp N}{N} \pm 1\right) = (fg)'(\pm 1) + O(N^{-1}) = O(N^{-1}),
\end{equation*}
uniformly for $x \in  \supp(Q)$. Hence, the sum \eqref{eq:2l} is realized effectively over finitely many $j$, and each term is of order $O(N^{-1})$. This leads to $\lr{V_\# f, \overline{g}} = O(N^{-3})$.
\end{proof}

\begin{lem}\label{lem:1h} Assume that $d \geq 3$ or that $d=1$ and $\int_\R q(x)dx = 0$. Define
\begin{equation}\label{eq:3y}
L \de \dfrac{1}{(2\pi)^d}\int_{\R^d} \dfrac{\hq(\xi) \hq(-\xi)}{|\xi|^2} d\xi \cdot \int_{[-1,1]^d} f(x)g(x) dx
\end{equation}
Then, there exists $r_0 > 0$ such that for any $\lambda \in \Dd(\lambda_0,r_0)$,
\begin{equation*}
\Pp\left( \left|N^2 \lr{L^{\lambda_0}_{q_0}(\lambda) V_\# f, \overline{V_\# g}} - L \right| \geq 2t \right) \leq C \exp\left(-\dfrac{ctN^{1/2}}{\ln(N)} \right).
\end{equation*}
\end{lem}

\begin{proof} In the Steps 1 to 7 below, we assume that $\lambda_0 \neq 0$ if $d=1$. In the Step 8, we deal with the case $\lambda = 0$ and $d=1$ -- which requires a special (though simpler) treatment.

\textbf{Step 1.} Fix $r_0 > 0$ such that $q_0$ has no resonance on $\Dd(\lambda_0,r_0) \setminus \lambda_0$. In order to estimate $\lr{L^{\lambda_0}_{q_0}(\lambda) V_\# f, \overline{V_\# g}}$, we first use \eqref{eq:2c}:
\begin{equation*}
L^{\lambda_0}_{q_0}(\lambda) = \dfrac{1}{2\pi i} \oint_{\lambda_0} \dfrac{R_{q_0}(\mu)}{\mu-\lambda} d\mu
\end{equation*}
(the contour integral is realized over $\p\Dd(\lambda_0,r_0)$). We combine the identities $R_{q_0}(\mu) = R_0(\mu) \left( \Id + q_0 R_0(\mu) \rho\right)^{-1} \left(\Id - q_0 R_0(\mu) (1-\rho) \right)$ with $(1-\rho) V_\# = 0$ to obtain $R_{q_0}(\mu) V_\# = R_0(\mu) V_\# + A(\mu)$, where
\begin{equations*}
A(\mu) \de - R_0(\mu) \left( \Id + q_0 R_0(\mu) \rho\right)^{-1} q_0 R_0(\mu) V_\#.
\end{equations*}
It follows that
\begin{equations}\label{eq:4a}
\lr{ L^{\lambda_0}_{q_0}(\lambda) V_\# f, \overline{V_\# g}} = \dfrac{1}{2\pi i} \oint_{\lambda_0} \dfrac{\lr{ R_0(\mu)V_\# f, \overline{V_\# g}}}{\mu-\lambda} d\mu  - \dfrac{1}{2\pi i} \oint_{\lambda_0} \dfrac{ \lr{A(\mu)  f, \overline{V_\# g}}}{\mu-\lambda} d\mu \\
 = \lr{R_0(\lambda) V_\# f, \overline{V_\# g}} -  \dfrac{1}{2\pi i} \oint_{\lambda_0} \dfrac{ \lr{A(\mu)  f, \overline{V_\# g}}}{\mu-\lambda} d\mu.
\end{equations}

\textbf{Step 2.} Since resonances of $q_0$ form a discrete set, $q_0$ has no resonances on a sufficiently small punctured neighborhood $U$ of $\lambda$. Therefore, the operator $\Id + q_0 R_0(\mu) \rho$ is invertible on $U$. Its family of $L^2$-inverses is analytic, hence by the Banach--Steinhauss theorem their operator norms are uniformly bounded on compact subsets of $U$. In addition, 
\begin{equation*}
\left( \Id + q_0 R_0(\mu) \rho\right)^{-1} q_0 = \rho (\Id + q_0 R_0(\mu) \rho)^{-1} q_0
\end{equation*}
(this can be checked expanding $\left( \Id + q_0 R_0(\mu) \rho\right)^{-1}$ by a Neumann series for $\Im \mu \gg 1$, and meromorphic continuation for $\mu \in \C \setminus \Res(q_0)$, see for instance the proof of \eqref{eq:0d}). Therefore,
\begin{equation}\label{eq:1p}
\left| \dfrac{1}{2\pi i} \oint_{\lambda_0} \dfrac{ \lr{A(\mu)  f, \overline{V_\# g}}}{\mu-\lambda} d\mu \right| \leq C \sup_{\mu \in \p\Dd(\lambda,r)}|\rho R_0(\mu) V_\# f||\rho R_0(\mu)^* V_\# g|.
\end{equation}
Since $R_0(\mu)$ and its adjoint $R_0(\mu)^* = R_0(-\overline{\mu})$ map $H^{-2}$ to $L^2$, since $f$ and  $g$ are smooth functions, the right hand side of \eqref{eq:1p} is controlled by $C |V_\#|_{\HH^{-2}}^2$. By Lemma \ref{lem:1d}, 
\begin{equation*}
\Pp\left( N^2\left|\dfrac{1}{2\pi i} \oint_{\lambda_0} \dfrac{ \lr{A(\mu)  f, \overline{V_\# g}}}{\mu-\lambda} d\mu \right| \geq t \right) \leq  \Pp(CN^2|V_\#|^2_{\HH^{-2}} \geq t) \leq Ce^{-ctN^{2\gamma-1}} \leq Ce^{-ctN}.
\end{equation*}
We used that the under the assumptions of Lemma \ref{lem:1h}, $\gamma \geq 3/2$. This estimate and \eqref{eq:4a} imply
\begin{equations*}
\Pp\left( \left|N^2 \lr{L^{\lambda_0}_{q_0}(\lambda) V_\# f, \overline{V_\# g}} - L \right| \geq 2t \right) 
\leq \Pp\left( \left|N^2 \lr{R_0(\lambda) V_\# f, \overline{V_\# g}} -L \right| \geq t \right) + O(e^{-ctN}).
\end{equations*}
Hence, with high probability and for $\lambda \in \Dd(\lambda_0,r_0)$, the terms $N^2 \lr{L^{\lambda_0}_{q_0}(\lambda) V_\# f, \overline{V_\# g}}$ and $\lr{R_0(\lambda) V_\# f, \overline{V_\# g}}$ are comparable. In contrast with $N^2 \lr{L^{\lambda_0}_{q_0}(\lambda) V_\# f, \overline{V_\# g}}$, the term $N^2 \lr{R_0(\lambda) V_\# f, \overline{V_\# g}}$ is meromorphic on the domain of holomorphy $X_d$ of $R_0(\lambda)$: 
\begin{equation*}
X_d = \C \text{ if } d \geq 3, \ \ X_d = \C \setminus 0 \text{ if } d = 1.
\end{equation*}
We write
\begin{equation*}
\lr{R_0(\lambda) V_\# f, \overline{V_\# g}} = \sum_{j,\ell} u_j u_\ell \alpha_{j\ell}(\lambda), \ \ \ \
\alpha_{j\ell}(\lambda) \de \lr{R_0(\lambda) q(N\cdot-j) f, \overline{q(N\cdot-\ell) g}}.
\end{equation*}
In the Steps 3, 4 and 5 below, we estimate the terms $\az_{j\ell}(\lambda)$ for $\Im \lambda \geq 1$, $\Im \lambda \leq -1$ and $|\Im \lambda| \leq 1$, respectively.

\textbf{Step 3.} We assume that $\Im \lambda \geq 1$. In this case, the operator $R_0(\lambda) = (-\Delta - \lambda^2)^{-1}$ is a Fourier multiplier with symbol $(|\xi|^2 - \lambda^2)^{-1}$. Using the Plancherel's identity and the substitutions $x \mapsto (x+j)/N$, $y \mapsto (y+\ell)/N$ and $\xi \mapsto N\xi$, we obtain
\begin{equations*}
\alpha_{j\ell}(\lambda) = \dfrac{1}{(2\pi)^d}\int_{\R^d} \dfrac{1}{|\xi|^2 - \lambda^2} \int_{\R^d} e^{-i\xi x} q(Nx-j) f(x) dx \int_{\R^d} e^{i\xi y} q(Ny-\ell) g(y) dy  d\xi \\
 = \dfrac{1}{(2\pi N)^d}\int_{\R^d} \dfrac{e^{i\xi (\ell-j)} \zeta_{j\ell}(\xi)}{N^2|\xi|^2 - \lambda^2} d\xi, \ \ \ \ \zeta_{j\ell}(\xi) \de \int_{\R^d} e^{-i\xi x} q(x) f\left( \dfrac{x+j}{N} \right) dx \int_{\R^d} e^{i\xi y} q(y) g\left( \dfrac{y+ \ell}{N} \right) dy.
\end{equations*} 
The function $\zeta_{j\ell}$ is the product of Fourier transforms of functions whose derivatives are all bounded uniformly in $x, j, \ell, N$. It follows that
\begin{equation}\label{eq:1k}
\zeta_{j\ell}(\xi) = O(\lr{\xi}^{-\infty}), \ \ \ \ \p_\xi \zeta_{j\ell}(\xi) = O(\lr{\xi}^{-\infty}).
\end{equation}
If in addition $d=1$, then $\int_{\R} q(x) dx = \hq(0) = 0$ according to the assumptions of the statement. Therefore, we can use the Taylor expansion \eqref{eq:3r} to obtain 
\begin{equation*}
d=1 \ \Rightarrow \ \int_{\R} e^{-ix\xi} q(x) u\left( \dfrac{x+j}{N} \right) dx = u\left(\dfrac{j}{N}\right) \hq(\xi) + O(N^{-1}) = O(|\xi| + N^{-1}),
\end{equation*}
uniformly for $\xi$ in bounded sets, $N \geq 1$ and $j \in [-N,N]$.
Using that a similar estimate is avalaible when $u$ is replaced by $v$ and $j$ is replaced by $\ell$, we get
\begin{equations}\label{eq:1l}
d=1 \ \Rightarrow \ |\zeta_{j\ell}(\xi)| =  O(|\xi|^2+N^{-2}) \ \ \ \ |\p_\xi \zeta_{j\ell}(\xi)| = O(|\xi| + N^{-1}).
\end{equations}

We now estimate  $\az_{jj}(\lambda)$. When $d \geq 3$, we split the integral over $\xi$ near $0$ and $\xi$ away from $0$ and we use the bounds \eqref{eq:1k} on $\zeta_{jj}(\xi)$:
\begin{equations}\label{eq:1m}
|\alpha_{jj}(\lambda)| \leq \dfrac{1}{(2\pi N)^d}\int_{\R^d} \dfrac{|\zeta_{jj}(\xi)|}{|N^2|\xi|^2 - \lambda^2|} d\xi 
\leq \dfrac{C}{N^d}\int_{|\xi| \leq 1} \dfrac{|\zeta_{jj}(\xi)|}{|N^2|\xi|^2 - \lambda^2|} d\xi + \dfrac{1}{N^d}\int_{|\xi| \geq 1} \dfrac{\lr{\xi}^{-2d}}{|N^2|\xi|^2 - \lambda^2|} d\xi \\
\leq \dfrac{C}{N^d}\int_{r=0}^1 \dfrac{\sup_{\w \in \Ss^{d-1}} |\zeta_{jj}(r\w)|}{|N^2 r^2 - \lambda^2|} r^{d-1} dr + \dfrac{C}{N^d}\int_{r=1}^\infty \dfrac{\lr{r}^{-d-1}}{|N^2r^2 - \lambda^2|} dr \\
\leq \dfrac{C}{N^{2d}}\int_{r=0}^N \dfrac{\sup_{\w \in \Ss^{d-1}} |\zeta_{jj}(r\w/N)|}{|r^2 - \lambda^2|} r^{d-1} dr + \dfrac{C}{N^d}\int_{r=1}^\infty \dfrac{\lr{r}^{-d-1}}{|N^2r^2 - \lambda^2|} dr \\
\leq \dfrac{C \lr{\lambda}}{N^{2d}}\int_{r=0}^N \dfrac{\sup_{\w \in \Ss^{d-1}} |\zeta_{jj}(r\w/N)|}{r^2+1} r^{d-1} dr + \dfrac{C\lr{\lambda}}{N^{d+2}}\int_{r=1}^\infty \lr{r}^{-d-1} dr.
\end{equations}
In the last line, we used the bound
\begin{equation}\label{eq:1r}
\Im \lambda \geq 1 \ \ \Rightarrow \ \ \sup_{r \geq 0}\dfrac{r^2 + 1}{|r^2-\lambda^2|} \leq C \lr{\lambda}.
\end{equation}
The second integral in the last line of \eqref{eq:1m} is finite and the contribution of the corresponding term is $O(\lr{\lambda} N^{-d-2})$. In dimension $d \geq 3$, we control the first integral by observing that the function $r \mapsto r^{d-1} \lr{r}^{-2}$ grows like $r^{d-3}$, which is not integrable. We deduce from \eqref{eq:1k} that
\begin{equation}\label{eq:4b}
d \geq 3 \ \Rightarrow \ \dfrac{C \lr{\lambda}}{N^{2d}}\int_{r=0}^N \dfrac{\sup_{\w \in \Ss^{d-1}} |\zeta_{jj}(r\w/N)|}{r^2+1} r^{d-1} dr \leq \dfrac{C \lr{\lambda}}{N^{d+2}}.
\end{equation}
When $d=1$, the estimate \eqref{eq:1l} implies
\begin{equations*}
\dfrac{C \lr{\lambda}}{N^2}\int_{r=0}^N \dfrac{\sup_{\w \in {\pm 1}}|\zeta_{jj}(r\w/N)|}{r^2+1} dr \leq \dfrac{C \lr{\lambda}}{N^2}  \int_{r=0}^N \dfrac{r^2N^{-2} + N^{-2}}{r^2+1} dr \leq \dfrac{C \lr{\lambda}}{N^3} = \dfrac{C}{N^{1+2}}.
\end{equations*}
This proves the validity of \eqref{eq:4b} for all $d \geq 1$. Combining the above estimates together, we obtain that for any odd $d$, uniformly in $\lambda$ with $\Im \lambda \geq 1$, $N$ and $j \in [-N,N]^d$,
\begin{equation*}
|\az_{jj}(\lambda)| \leq \dfrac{C\lr{\lambda}}{N^{d+2}}.
\end{equation*}

We now estimate $\az_{j\ell}(\lambda)$, for $j \neq \ell$. We first integrate by parts in $\xi$, using the identity $(j-\ell) D_\xi e^{i \xi (\ell-j)} = |j-\ell|^2 e^{i \xi (\ell-j)}$:
\begin{equations}\label{eq:1n}
\az_{j\ell}(\lambda) = \dfrac{1}{(2\pi N)^d}\int_{\R^d} \dfrac{(j-\ell) D_\xi e^{i \xi (\ell-j)}}{|j-\ell|^2} \cdot \dfrac{\zeta_{j\ell}(\xi)}{N^2|\xi|^2 - \lambda^2}   d\xi \\
 = \dfrac{1}{(2\pi N)^d |j-\ell|^2}\int_{\R^d} e^{i \xi (\ell-j)} \left( \dfrac{(j-\ell)D_\xi \zeta_{j\ell}(\xi)}{N^2 |\xi|^2-\lambda^2} -\dfrac{2N^2(j-\ell)\xi }{i(N^2|\xi|^2-\lambda^2)^2} \zeta_{j\ell}(\xi) \right) d\xi.
\end{equations}
We argue as in the case $j= \ell$ to bound the first term. This yields 
\begin{equations*}
\left|\dfrac{1}{(2\pi N)^d |j-\ell|^2}\int_{\R^d} e^{i \xi (\ell-j)} \dfrac{(j-\ell)D_\xi \zeta_{j\ell}(\xi)}{N^2 |\xi|^2-\lambda^2} d\xi \right| \\ 
\leq \dfrac{C \lr{\lambda}}{N^{2d} |j-\ell|}\int_{r=0}^N \dfrac{\sup_{\w \in \Ss^{d-1}} |D_\xi \zeta_{j\ell}(r\w/N)|}{r^2+1} r^{d-1} dr + \dfrac{C\lr{\lambda}}{N^{d+2} |j-\ell|}\int_{r=1}^\infty \lr{r}^{-d-1} dr.
\end{equations*} 
The second integral contributes to $O(\lr{\lambda} N^{-d-2} |j-\ell|^{-1})$. So does the first, when $d \geq 3$. When $d=1$, we use \eqref{eq:1l} to obtain
\begin{equation*}
\dfrac{C \lr{\lambda}}{N^2 |j-\ell|}\int_{r=0}^N \dfrac{\sup_{\w \in \{\pm 1\}} |D_\xi \zeta_{j\ell}(r\w/N)|}{r^2+1} dr \leq \dfrac{C \lr{\lambda}}{N^2 |j-\ell|}\int_{r=0}^N \dfrac{N^{-1}r + N^{-1}}{r^2+1} dr \leq  \dfrac{C \lr{\lambda} \ln(N)}{N^3|j-\ell|}.
\end{equation*}
In order to control the second term in the second line of \eqref{eq:1n}, we use techniques similar to the case $j=\ell$ and obtain
\begin{equations*}
\left| \dfrac{C}{N^d|j-\ell|^2}\int_{\R^d} e^{i \xi (\ell-j)} \dfrac{2N^2(j-\ell)\xi }{i(N^2|\xi|^2-\lambda^2)^2} \zeta_{j\ell}(\xi)  d\xi \right| \\
\leq \dfrac{C\lr{\lambda}^2}{N^{2d-1}|j-\ell|}\int_{r = 0}^N  \dfrac{r^d \sup_{\w \in \Ss^{d-1}} |\zeta_{j\ell}(r\w/N)|}{(1+r^2)^2}   dr + \dfrac{C \lr{\lambda}^2}{N^{d+2}|j-\ell|} \int_{r=1}^\infty \lr{r}^{-d-1} dr
\end{equations*} 
The second integral is $O(\lr{\lambda}^2 N^{-d-2}|j-\ell|^{-1})$. In dimension $d \geq 5$, given that the function $r^d\lr{r}^{-4}$ is not integrable, the first integral can be controlled by $O(\lr{\lambda}^2 N^{-d-2} |j-\ell|^{-1})$. For the same reason, in dimension $3$, it can be controlled by $O(\lr{\lambda}^2 N^{-5} \ln(N) |j-\ell|^{-1})$. In dimension one, we use \eqref{eq:1l} to obtain
\begin{equation*}
\dfrac{C\lr{\lambda}^2}{N |j-\ell|}\int_{r = 0}^N  \dfrac{r \sup_{\w \in \{\pm 1\}} |\zeta_{j\ell}(r\w/N)|}{(1+r^2)^2}   dr \leq \dfrac{C\lr{\lambda}^2}{N^3|j-\ell|}\int_{r = 0}^N  \dfrac{r}{1+r^2}   dr \leq \dfrac{C\lr{\lambda}^2 \ln(N)}{N^3|j-\ell|}.
\end{equation*}
Grouping all these estimates together, we obtain that for any odd $d$, uniformly in $\lambda$ with $\Im \lambda \geq 1$, $N$ and $j \neq \ell \in [-N,N]^d$,
\begin{equation}\label{eq:2m}
|\az_{j\ell}(\lambda)| \leq \dfrac{C \lr{\lambda}^2 \ln(N)}{|j-\ell| N^{d+2}}.
\end{equation}

\textbf{Step 4.} We show here estimates for $\Im \lambda \leq -1$. In such cases, $\Im(-\lambda) \geq 1$, and we can write
\begin{equation*}
\az_{j\ell}(\lambda) = \az_{j\ell}(-\lambda) + \lr{(R_0(\lambda)-R_0(-\lambda))q(N\cdot-j) f, \overline{q(N\cdot-\ell) g}}.
\end{equation*}
The same arguments as in \cite[Lemma 4.9]{Dr2} shows that $\rho(R_0(\lambda)-R_0(-\lambda))\rho$ maps $L^2$ to $H^{2d}$, with norm controlled by  $C|\lambda|^{3d-2} e^{C|\Im \lambda|}$ -- when $\-\Im \lambda \leq -1$. By duality, it also maps $H^{-2d}$ to $L^2$ (with same norm) and by interpolation, $H^{-d}$ to $H^d$ (with same norm). It follows that
\begin{equations*}
\left|\blr{(R_0(\lambda)-R_0(-\lambda))q(N\cdot-j) f, \overline{q(N\cdot-\ell) g}}\right| \\ \leq C|\lambda|^{3d-2} e^{C|\Im \lambda|} |q(N\cdot-j) f|_{H^{-d}} |q(N\cdot-\ell) g|_{H^{-d}}.
\end{equations*}
Since $H^d$ is an algebra, the dual bound $|f_1 f_2|_{H^{-d}} \leq C|f_1|_{H^d} |f_2|_{H^{-d}}$ holds for arbitrary $f_1, f_2 \in C_0^\infty(\R^d,\C)$. Since $f$ and $g$ are smooth and $q$ has compact support,
\begin{equation*}
|q(N\cdot-j) f|_{H^{-d}} |q(N\cdot-\ell) g|_{H^{-d}} \leq C |q(N \cdot)|_{H^{-d}}^2.
\end{equation*}
A computation shows
\begin{equation*}
|q(N \cdot)|_{H^{-d}}^2 = \dfrac{1}{(2\pi N)^{d}}\int_{\R^d} \dfrac{|\hq(\xi)|^2}{(1+N^2\xi^2)^{2d}} d\xi \leq \dfrac{C}{N^{2d}}\int_0^\infty \dfrac{\sup_{\w \in \Ss^{d-1}}|\hq(r\w)|^2}{(1+N^2r^2)^{2d}} r^{d-1}dr.
\end{equation*}
Again, splitting this integral for $r$ near $0$ and $r$ away from $0$, and using $\hq(0) = 0$ in dimension one, shows that $|q(N \cdot)|_{H^{-d}}^2 = O(N^{-d-3})$. Therefore, 
\begin{equation}\label{eq:0c}
|\az_{j\ell}(\lambda)-\az_{j\ell}(-\lambda)| \leq \dfrac{C |\lambda|^{3d-2} e^{C|\Im \lambda|}}{N^{d+3}}.
\end{equation}
In particular, this bound combined with \eqref{eq:2m} shows that for $|\Im \lambda| = 1$, 
\begin{equation}\label{eq:1s}
|\az_{jj}(\lambda)| \leq  \dfrac{C\lr{\lambda}^{3d-2}}{N^{d+2}}, \ \ |\az_{j\ell}(\lambda)| \leq \dfrac{C \lr{\lambda}^2 \ln(N)}{|j-\ell| N^{d+2}} + \dfrac{C |\lambda|^{3d-2}}{N^{d+3}}.
\end{equation}

\textbf{Step 5.} We now use the three lines theorem to estimate locally $\az_{j\ell}(\lambda)$, for any $\lambda \in X_d$. Because of the conclusions of Steps 3 and 4, it suffices to focus on the strip $|\Im \lambda| \leq 1$. For $d \geq 3$, the function $\az_{j\ell}(\lambda)$ is bounded in this strip:
\begin{equation*}
|\az_{j\ell}(\lambda)| = |\lr{R_0(\lambda)q(N\cdot-j) f, \overline{q(N\cdot-\ell) g}}| \leq C e^{C|\Im \lambda|} |q|_\infty^2 \cdot \sup_{x,y \in \supp(q)^2} |u(x)||v(y)|.
\end{equation*}
The function $(\lambda+2i)^{-3d} \az_{j\ell}(\lambda)$ is also bounded in the strip $|\Im \lambda| \leq 1$, and \eqref{eq:1s} estimates it on the edge of this strip. The three lines theorem imply that the bound \eqref{eq:1s} holds uniformly inside the strip $|\Im \lambda| \leq 1$.

When $d=1$, we remove the pole of $R_0(\lambda)$ at $0$ by considering the function $\lambda \az_{j\ell}(\lambda)$ instead. The same arguments as in the case $d \geq 3$ extends \eqref{eq:1s} to $|\Im \lambda| \leq 1$, $\lambda \neq 0$. In particular, for any $\lambda \in X_d$, there exists $C > 0$ such that uniformly in $j, \ell, N$,
\begin{equation*}
|\az_{jj}(\lambda)| \leq  \dfrac{C}{N^{d+2}}, \ \ |\az_{j\ell}(\lambda)| \leq \dfrac{C\ln(N)}{|j-\ell| N^{d+2}} + \dfrac{C}{N^{d+3}}.
\end{equation*}

\textbf{Step 6.} We now estimate the HS-norm of $\az_{j\ell}(\lambda)$ for $\lambda$ in compact subsets of $X_d$, so that we can apply later the Hanson--Wright inequality. According to Step 5, it suffices to estimate the sum
\begin{equation*}
\sum_j \dfrac{1}{N^{2d+4}} + \sum_{j\neq \ell} \dfrac{\ln(N)^2}{N^{4+2d} |j-\ell|^2} + \sum_{j \neq \ell} \dfrac{1}{N^{2d+6}}.
\end{equation*}
Given $m$, the number of sites $j,\ell$ such that $|j-\ell| = m$ is controlled by $N^{2d-1}$. Therefore,
\begin{equations*}
\sum_j \dfrac{1}{N^{2d+4}} + \sum_{j\neq \ell} \dfrac{\ln(N)^2}{N^{4+2d} |j-\ell|^2} + \sum_{j \neq \ell} \dfrac{1}{N^{2d+6}} \\ \leq \dfrac{N^d}{N^{2d+4}} +  \dfrac{CN^{2d-1} \ln(N)^2}{N^{4+2d}} \sum_{m=1}^{2N} \dfrac{1}{m^2} + \dfrac{N^{2d}}{N^{2d+6}} \leq \dfrac{C\ln(N)^2}{N^5}. 
\end{equations*}
It follows that $|\az(\lambda)|_{\HS} = O(N^{-5/2} \ln(N))$. By the Hanson--Wright inequality (in its original version), for $\lambda$ in compact subsets of $X_d$,
\begin{equation}\label{eq:1t}
\Pp\left(N^2\left|\lr{R_0(\lambda) V_\# f, V_\# g} -  \sum_j \az_{jj}(\lambda)\right| \geq t\right) \leq C \exp \left( -\dfrac{c t N^{1/2}}{\ln(N)} \right).
\end{equation}

\textbf{Step 7.} To conclude we estimate $\sum_j \az_{jj}(\lambda)$. For $\Im \lambda \geq 1$, we have
\begin{equation*}
\sum_j \az_{jj}(\lambda) = \dfrac{1}{(2\pi N)^d} \sum_j \int_{\R^d} \dfrac{\zeta_{jj}(\xi)}{N^2|\xi|^2 - \lambda^2} d\xi.
\end{equation*}
The bound \eqref{eq:3r} shows that uniformly in $j$,
\begin{equation*}
\zeta_{jj}(\xi) = \hq(\xi)\hq(-\xi) \cdot (fg)\left(\dfrac{j}{n}\right) + O(N^{-1} \lr{\xi}^{-2d}).
\end{equation*}
It follows that
\begin{equation*}
\sum_j \az_{jj}(\lambda) = \dfrac{1}{(2\pi)^d}\int_{\R^d} \dfrac{\hq(\xi)\hq(-\xi)}{N^2|\xi|^2 - \lambda^2} d\xi \cdot \dfrac{1}{N^d}\sum_j (fg)\left(\dfrac{j}{n}\right) + O(N^{-1}) \int_{\R^d} \dfrac{\lr{\xi}^{-2d}}{|N^2|\xi|^2 - \lambda^2|} d\xi.
\end{equation*}
We recognize a Riemann sum with step $N^{-1}$ on the right hand side. Since the function $fg$ is smooth, this Riemann sum is equal to $\int_{[-1,1]^d} (fg)(x) dx$ modulo $O(N^{-1})$. In addition, we can use \eqref{eq:1r} to control the second term, and obtain
\begin{equation*}
\sum_j \az_{jj}(\lambda) = \dfrac{1}{(2\pi)^d}\int_{\R^d} \dfrac{\hq(\xi) \hq(-\xi)}{N^2|\xi|^2 - \lambda^2} d\xi \cdot \int_{[-1,1]^d} (fg)(x) dx + O(N^{-3} \lr{\lambda}).
\end{equation*}
To obtain an asymptotic of the right hand side, we observe that the function $\hq(\xi) \hq(-\xi)/|\xi|^2$ is integrable (because $\hq(0) = 0$ when $d=1$ by assumption), and
\begin{equation*}
\left|\int_{\R^d} \dfrac{\hq(\xi) \hq(-\xi)}{N^2|\xi|^2 - \lambda^2} d\xi - \int_{\R^d} \dfrac{\hq(\xi) \hq(-\xi)}{N^2|\xi|^2} d\xi\right| \leq \int_{\R^d} \dfrac{|\hq(\xi) \hq(-\xi)|  |\lambda|^2}{N^2 |\xi|^2 \cdot |N^2|\xi|^2 - \lambda^2|} d\xi.
\end{equation*}
We apply \eqref{eq:1r} to control the LHS:
\begin{equation*}
\int_{\R^d} \dfrac{|\hq(\xi) \hq(-\xi)| |\lambda|^2}{N^2 |\xi|^2 \cdot |N^2|\xi|^2 - \lambda^2|} d\xi \leq \int_{\R^d} \dfrac{|\hq(\xi) \hq(-\xi)| \lr{\lambda}^3}{N^2 |\xi|^2 \cdot (N^2|\xi|^2+1)} d\xi.
\end{equation*}
As in Step 3, we can split this integral in a part near $\xi=0$ and a part away from $\xi=0$. Using that $\hq$ has fast decay (and $\hq(0) = 0$ when $d=1$), an upper bound is $O(N^{-3} \lr{\lambda}^3)$, therefore $\sum_j \az_{jj}(\lambda) = $
\begin{equation*}
\dfrac{1}{(2\pi)^d N^2}\int_{\R^d} \dfrac{\hq(\xi) \hq(-\xi)}{|\xi|^2} d\xi \cdot \int_{[-1,1]^d} (fg)(x) dx + O(N^{-3} \lr{\lambda}^3) = \dfrac{L}{N^2} + O(N^{-3} \lr{\lambda}^3),
\end{equation*}
where $L$ was defined in \eqref{eq:3y}. Thanks to \eqref{eq:0c}, this estimate generalize to $\Im \lambda \leq -1$, and by the same arguments as in Step 5, for any $\lambda \in X_d$. Now, \eqref{eq:1t} yields that for $\lambda$ in compact subsets of $X_d$,
\begin{equation*}
\Pp\left(\left|N^2\lr{R_0(\lambda) V_\# f, \overline{V_\# g}} -   L \right| \geq t \right) \leq C\exp\left( -\dfrac{ctN^{1/2}}{\ln(N)} \right).
\end{equation*}

\textbf{Step 8.} Here we deal with the case $d=1$ and $\lambda_0=0$. Step 1 goes through and yields
\begin{equation*}
\lr{L^0_{q_0}(\lambda) V_\# f, \overline{V_\# g}} = \dfrac{1}{2\pi i} \oint_0 \dfrac{\lr{ R_0(\mu)V_\# f, \overline{V_\# g}}}{\mu-\lambda} d\mu  - \dfrac{1}{2\pi i} \oint_0 \dfrac{ \lr{A(\mu) f, \overline{V_\# g}}}{\mu-\lambda} d\mu.
\end{equation*}
Since $R_0(\mu)$ has a simple pole at $0$, we obtain
\begin{equation*}
\lr{L^0_{q_0}(\lambda) V_\# f, \overline{V_\# g}} = \lr{ L^0_0(0)V_\# f, \overline{V_\# g}} - \dfrac{1}{2\pi i} \oint_0 \dfrac{ \lr{A(\mu)  f, \overline{V_\# g}}}{\mu-\lambda} d\mu.
\end{equation*}
The same method as in Step 2 above shows that
\begin{equation*}
\Pp\left( N^2\left|\dfrac{1}{2\pi i} \oint_{\lambda_0} \dfrac{ \lr{A(\mu)  f, \overline{V_\# g}}}{\mu-\lambda} d\mu \right| \geq t \right) \leq Ce^{-ctN}.
\end{equation*}

\newcommand{\tK}{{\tilde{K}}} We now need to estimate $N^2 \lr{L^0_0(\lambda) V_\# f, \overline{V_\# g}}$. The kernel of $L^0_0(\lambda)$ is explicitly given by $(x,y) \mapsto K(\lambda, |x-y|)$, where $K(\lambda, r) = \frac{i}{2\lambda} (e^{i\lambda r}-1)$ -- see for instance \cite[(2.2.1)]{DZ}. Therefore, we can write 
\begin{equation*}
\lr{L^0_0(\lambda) V_\# f, \overline{V_\# g}} = \sum_{j,\ell} \beta_{j\ell} u_ju_\ell, \ \ \beta_{j\ell}(\lambda) \de \int_{\R^2} K(\lambda,|x-y|) q(Nx-j) q(Ny-\ell) f(x) g(y) dx dy.
\end{equation*}
We estimate the term $\beta_{j\ell}(\lambda)$. For this purpose, we define
\begin{equation*}
\gamma[\tK,q_1,q_2,\tu,\tw] \de \dfrac{1}{N^2} \int_{\R^2} \tK \left(\lambda,\left|\dfrac{x+j}{N} - \dfrac{y+\ell}{N}\right|\right) q_1(x) q_2(y) \tu \left(\dfrac{x+j}{N}\right) \tw\left( \dfrac{y+\ell}{N} \right) dx dy,
\end{equation*}
where $q_1$, $q_2$ are smooth with compact support, $\tu$, $\tw$ are $C^\infty$ functions on $\R$, and $\tK : \C \times  \R \rightarrow \C$ is holomorphic in the first variable and locally bounded in the second. In this context, it is clear that $\gamma[\tK, q_1, q_2, \tu, \tw]$ is $O(N^{-2})$. We observe that $\beta_{j\ell}(\lambda) = \gamma[f,q,q,u,v]$. Let $Q$ be the compactly supported antiderivative of $q$. Using that $Q' = q$, we can integrate by parts $\gamma[f,q,q,u,v]$ in $x$, and get:
\begin{equation*}
\beta_{j\ell}(\lambda) = -\dfrac{\gamma[\sgn \cdot \p_2 f,Q,q,u,v] + \gamma[f,Q,q,u',v]}{N}
\end{equation*}
The function $\sgn \cdot \p_2 f$ is locally bounded but it has a discontinuity at $0$. Its derivative is the distribution $2 \delta_0 \cdot \p_2 f$. Therefore, an integration by parts in $y$ generates a boundary term:
\begin{equations*}
\beta_{j\ell}(\lambda) = -\dfrac{2}{N^3} \int_\R \p_2 K(\lambda,0) Q(x) Q(x+j-\ell) f \left(  \dfrac{x+j}{N}\right) g \left( \dfrac{y+\ell}{N} \right) dx \\ 
+\dfrac{\gamma[\p_2^2 f,Q,Q,u,v]+\gamma[\sgn \cdot \p_2 f,Q,Q,u,v'] + \gamma[\sgn \cdot \p_2 f,Q,Q,u',v] +  \gamma[f,Q,Q,u',v']}{N^2}
\end{equations*}
Because of $\gamma[\tK, q_1, q_2, \tu, \tw] = O(N^{-2})$, of $\p_2 K(\lambda, 0) =  -\frac{1}{2}$ and of \eqref{eq:3r},
\begin{equations}\label{eq:2o}
\beta_{j\ell}(\lambda) = \dfrac{1}{N^3} \int_\R Q(x) Q(x+j-\ell) dx \cdot f\left( \dfrac{j}{N} \right) g\left( \dfrac{\ell}{N} \right) + O(N^{-4}).
\end{equations}
When $|j-\ell|$ is large enough, $Q$ and $Q(\cdot+j-\ell)$ have disjoint support. Therefore, the leading term in \eqref{eq:2o} vanishes unless $j-\ell$ is smaller than a constant independent of $N$.  This yields the estimate:
\begin{equation*}
|\beta|_\HS^2 = O(N) O(N^{-6}) + O(N^2) O(N^{-8}) = O(N^{-5}).
\end{equation*}

We conclude by estimating $\sum_j \beta_{jj}(\lambda)$: recognizing a Riemann sum,
\begin{equations*}
\sum_j \beta_{jj}(\lambda) = \dfrac{1}{N^2} \int_\R Q(x)^2 dx \cdot \dfrac{1}{N}\sum_j f\left( \dfrac{j}{N} \right) g\left( \dfrac{j}{N} \right) + O(N^{-3}) \\ = \dfrac{1}{N^2} \int_\R Q(x)^2 dx \cdot \int_{-1}^1 (fg)(x) dx + O(N^{-3}).
\end{equations*}
We conclude that
\begin{equation*}
\Pp(|N^2 \lr{L^0_{q_0}(\lambda) V_\# f, \overline{V_\# g}} - L| \geq t) \leq C e^{-ct N^{1/2}}.
\end{equation*}
Since $\hat{Q}(\xi) = \hq(\xi)/\xi$, the lemma for $\lambda = 0$ and $d=1$ follows.
\end{proof}

\section{Proofs of the theorems}\label{sec:4}

\subsection{Localization}

\begin{proof}[Proof of Theorem \ref{thm:1} and of Corollary \ref{cor:1}] Assume that $q_0$ has no resonances on $\p \Dd(0,R)$. According to Lemma \ref{lem:1b}, it suffices to show that after removing a set of probability $O(e^{-cN^\gamma})$, $|V_\#|_{\HH^{-2}} \leq c' N^{-\gamma/2}$, for some $c'$ sufficiently large. This follows from Lemma \ref{lem:1d}. This concludes the proof of Theorem \ref{thm:1}.

We now show Corollary \ref{cor:1}. Fix $R > 0$ such that $q_0$ has no resonances on $\p \Dd(0,R)$. Introduce the event:
\begin{equation*}
A_N \de \left\{\ V_N^\w \text{ does not satisfy } \eqref{eq:2q}    \right\}.
\end{equation*}
We know that $\Pp(A_N) \leq Ce^{-cN^\gamma}$. In particular, $\sum_{N=0}^\infty \Pp(A_N) < \infty$, and the Borel--Cantelli lemma implies that $A_N$ happens only finitely many times. Therefore, $\Pp$-almost surely, there exists $N_0$ such that for every $N \geq N_0$, \eqref{eq:2q} is realized. It suffices to take a countable sequence $R \rightarrow \infty$ to conclude.
\end{proof}

\begin{proof}[Proof of Theorem \ref{thm:0}] The proof of this theorem follows from Lemma \ref{lem:1i} and arguments from \cite{Dr2,Dr3}. Recall that $X_d = \C$ if $d \geq 3$ and $X_1 = \C \setminus 0$. According to the formula
\begin{equation*}
R_{V_N}(\lambda) = R_0(\lambda) (\Id + V_\# R_0(\lambda) \rho)^{-1} (\Id - V_\# R_0(\lambda) (1-\rho)),
\end{equation*}
the resonances of $V = V_\#$ in $X_d$ are exactly the poles of $(\Id + V_\# R_0(\lambda) \rho)^{-1}$. In addition, \cite[Theorem 2.8]{DZ} implies
\begin{equation*}
|(V_\# R_0(\lambda) \rho)^2|_\BB \leq |V_\#|_\infty |\rho R_0(\lambda) \rho|_{H^{-2} \rightarrow L^2} |V_\#|_{\HH^{-2}} |\rho R_0(\lambda) \rho|_{L^2 \rightarrow H^2} \leq \dfrac{Ce^{c(\Im \lambda)_-}}{d-1+|\lambda|^2} |V_\#|_{\HH^{-2}}.
\end{equation*}
If the RHS is bounded by $1/2$, then the operator $\Id + V_\# R_0(\lambda) \rho$ is invertible and $\lambda$ is not a resonance. According to Lemma \ref{lem:1i}, $|V_\#|_{\HH^{-2}} \leq N^{-\gamma}$ with probability $1-Ce^{-cN^\gamma}$. Therefore, if $\Im \lambda \geq (\ln(2C)-\gamma \ln(N))/c$ then with same probability,
\begin{equation*}
\dfrac{Ce^{c(\Im \lambda)_-}}{d-1+|\lambda|^2} |V_\#|_{\HH^{-2}} < 1/2.
\end{equation*}
Since Theorem \ref{thm:0} says something only for large values of $N$,  and since the only possible resonance near $0$ was localized in Theorem \ref{thm:1}, the proof is over.\end{proof}

\subsection{Finer estimates for simple resonances}

\begin{proof}[Proof of Theorem \ref{thm:2}] Fix $\lambda_0 \in \Res(q_0)$, of multiplicity $1$, with resonant states $f$ and $g$. Recall the estimate \eqref{eq:2y}: with probability $1-O(e^{-cN^{1/4}})$,
\begin{equations}\label{eq:3z}
|V_\#|_{\HH^{-1}} \leq N^{-3/8} \ \ \text{ if } d=1 \text{ and } \int_\R q(x) dx \neq 0, \\
|V_\#|_{\HH^{-1}} \leq N^{-7/8} \ \ \text{ if } d \geq 3; \text{ or } d=1 \text{ and } \int_\R q(x) dx = 0.
\end{equations}
Let us define the following event:
\begin{equations*}
B_N = \{ |V_\#|_{\HH^{-2}} \geq N^{-\gamma/2}\} \cup \{ V_\# \text{ does not satisfy  }\eqref{eq:3z} \}.
\end{equations*}
Because of Lemmas \ref{lem:1d} and \ref{lem:1l}, $\Pp(B_N) = O(e^{-cN^{1/4}})$. Let $\delta_0, r_0$ be given by Lemma \ref{lem:1c}. As in the proof of Theorem \ref{thm:1}, we know that for $N$ sufficiently large, $V_N$ has a single resonance $\lambda_N \in \Dd(\lambda_0,r_0)$, for the event $\Omega \setminus B_N$. We extend $\lambda_N$ to all of $\Omega$ by setting $\lambda_N|_{B_N} = \lambda_0$. This definition shows that the random variable $\lambda_N$ is a resonance of $V_N$ with probability $1-\Pp(B_N) = 1 - O(e^{-cN^{1/4}})$. Moreover, Lemma \ref{lem:1c} implies that $\lambda_N$ satisfies the equation
\begin{equation}\label{eq:3b}
\lambda_N - \lambda_0 = \1_{\Omega \setminus B_N} \sum_{k = 0}^\infty (-1)^{k+1} \blr{(V_\# L_{q_0}^{\lambda_0}(\lambda_N)\rho)^k V_\# f,\overline{g}}.
\end{equation}
as long as $N$ is sufficiently large. It remains to study the speed of convergence of $\lambda_N$, in the Cases I, II and III.

\textbf{Case I.} In this case, $d=1$ or $3$ and $\int_{\R^d} q(x) dx \neq 0$, hence $\gamma=d/2$. We have:
\begin{equations}\label{eq:3c}
\dfrac{N^{d/2}}{i}(\lambda_N-\lambda_0) = \1_{\Omega \setminus B_N} N^{d/2} \sum_{k = 0}^\infty (-1)^{k+1} \blr{(V_\# L_{q_0}^{\lambda_0}(\lambda_N)\rho)^k V_\# f,\overline{g}} \\
 = -N^{d/2}\lr{V_\# f,\overline{g}} + \1_{B_N} N^{d/2} \lr{V_\# f,\overline{g}} + \1_{\Omega \setminus B_N} N^{d/2} \sum_{k = 1}^\infty (-1)^{k+1} \blr{(V_\# L_{q_0}^{\lambda_0}(\lambda_N)\rho)^k V_\# f,\overline{g}}.
\end{equations}
We show that the first and second terms in the second line of \eqref{eq:3c} converge in $L^1$ to $0$. Using that $\Pp(B_N)$ decay exponentially and that $\lr{V_\# f,g} = O(1)$,
\begin{equation*}
\Ee(\1_{B_N} N^{d/2} |\lr{V_\# f,\overline{g}}|) = O(N^{d/2} e^{-cN^{1/4}}) \rightarrow 0.
\end{equation*}
In addition, on $\Omega \setminus B_N$, $N^{d/2}|V_\#|_{\HH^{-1}}^2 \leq N^{-1/4}$ by \eqref{eq:2y} (recall that $d=1$ or $3$). The estimate \eqref{eq:1d} and the definition of $B_N$ yields
\begin{equation*}
\left|\1_{\Omega \setminus B_N} N^{d/2} \sum_{k = 1}^\infty (-1)^{k+1} \blr{(V_\# L_{q_0}^{\lambda_0}(\lambda_N)\rho)^k V_\# f,\overline{g}} \right| \leq C \1_{\Omega \setminus B_N} N^{d/2} |V_\#|_{\HH^{-1}}^2) = O(N^{-1/4}).
\end{equation*}
We deduce that
\begin{equation*}
\dfrac{N^{d/2}}{i}(\lambda_N-\lambda_0) \law N^{d/2}\lr{V_\# f,\overline{g}}.
\end{equation*}
 -- see for instance \cite[Theorem 25.4]{Bi}. Lemma \ref{lem:1e} shows that $N^{d/2}\lr{V_\# f,\overline{g}}$ converges to a Gaussian and this concludes the proof of Theorem \ref{thm:2} in Case I.

\textbf{Case II.} In this case, $d=1$ and $\int_\R q(x) dx = 0$, $\int_\R x q(x) dx \neq 0$ and $(f \cdot g)' \not\equiv 0$ on $[-1,1]$. We use \eqref{eq:3c} with a factor $N^{3/2}$ instead of $N^{1/2}$: $\frac{N^{3/2}}{i}(\lambda_N-\lambda_0) = $
\begin{equations*}
-N^{3/2}\lr{V_\# f,\overline{g}} + \1_{B_N} N^{3/2} \lr{V_\# f,\overline{g}} + \1_{\Omega \setminus B_N} N^{3/2} \sum_{k = 1}^\infty (-1)^{k+1} \blr{(V_\# L_{q_0}^{\lambda_0}(\lambda_N)\rho)^k V_\# f,\overline{g}}.
\end{equations*}
The first term converges to a Gaussian according to Lemma \ref{lem:1g}. The second term converges in $L^1$ to $0$ because $\Pp(B_N)$ is exponentially small. The third term is $O(N^{-1/4})$ because it is bounded by $N^{3/2}|V_\#|_{\HH^{-1}}^2$ -- see \eqref{eq:1d} -- itself being $O(N^{-1/4})$ for events in $\Omega \setminus B_N$ -- see \eqref{eq:3z} and the definition of $B_N$. An application of \cite[Theorem 25.4]{Bi} as in Case I allows us to conclude.

\textbf{Case III.} Thanks to \eqref{eq:3b}, we can write
\begin{equations}\label{eq:3l}
\dfrac{N^2}{i}(\lambda_N-\lambda_0) = - \1_{\Omega \setminus B_N} N^2 \lr{V_\# f, \overline{g}} + \1_{\Omega \setminus B_N} N^2 \blr{V_\# L_{q_0}^{\lambda_0}(\lambda_N) V_\# f,\overline{g}} \\
  + \1_{\Omega \setminus B_N} N^2 \sum_{k = 2}^\infty (-1)^{k+1} \blr{(V_\# L_{q_0}^{\lambda_0}(\lambda_N) \rho)^k V_\# f,\overline{g}}.
\end{equations}
We now evaluate the probability that the RHS of \eqref{eq:3l} is significantly different from $L$ -- defined in Lemma \ref{lem:1h}. Since Case III is satisfied,
\begin{equation}\label{eq:3o}
\Pp(|N^2 \lr{V_\# f, \overline{g}}| \geq N^{-1/4}) = O(e^{-cN^{1/2}}).
\end{equation}
This comes from Lemma \ref{lem:1e} when $d \geq 5$; Lemma \ref{lem:1k} when $d=3$ and $\int_{\R^3} q(x) dx  = 0$ or $d=1$ and $\int_{\R} q(x) dx = \int_\R x q(x) dx = 0$; and Lemma \ref{lem:1g} when $\int_\R q(x) dx = 0$ and $(fg)' = 0$ on $[-1,1]$. According to Lemma \ref{lem:1h} and $\Pp(B_N) = O(e^{-cN^{1/4}})$,
\begin{equations}\label{eq:3n}
\Pp(|\1_{\Omega \setminus B_N} N^2 \blr{V_\# L_{q_0}^{\lambda_0}(\lambda_N) V_\# f,\overline{g}}-L| \geq N^{-1/5}) \\ \leq \Pp(2|N^2 \blr{V_\# L_{q_0}^{\lambda_0}(\lambda_N) V_\# f,\overline{g}}-L| \geq N^{-1/5}) +  \Pp(B_N) = O(e^{-cN^{1/4}}).
\end{equations}
Since Case III is satisfied, then $m \geq 1$ when $d=1$. Hence, $|V_\#|_{\HH^{-1}} = O(N^{-7/8})$ on $\Omega \setminus B_N$ -- see \eqref{eq:2y}. Thanks to \eqref{eq:1d},
\begin{equation}\label{eq:3m}
\left|\1_{\Omega \setminus B_N} N^2 \sum_{k = 2}^\infty (-1)^{k+1} \blr{(V_\# L_{q_0}^{\lambda_0}(\lambda_N) \rho)^k V_\# f,\overline{g}} \right| \leq C \1_{\Omega \setminus B_N} N^2 |V_\#|_{\HH^{-1}}^3 = O(N^{-1/4}).
\end{equation}

Combining \eqref{eq:3l}, \eqref{eq:3o}, \eqref{eq:3n} and \eqref{eq:3m}, we obtain
\begin{equation*}
\Pp(|N^2 (\lambda_N - \lambda_0) - i L| \geq N^{-1/5}) = O(e^{-cN^{1/4}}).
\end{equation*}
In particular,
\begin{equation*}
\sum_{N = 1}^\infty \Pp(|N^2 (\lambda_N - \lambda_0) - i L| \geq N^{-1/5}) < \infty.
\end{equation*}
This implies by the Borel--Cantelli lemma that for each elementary event, $|N^2(\lambda_N-\lambda_0)-L| \geq N^{-1/5}$ for only finitely many $N$. In particular, $N^2(\lambda_N-\lambda_0) \ras i L$ as claimed. \end{proof}

\begin{proof}[Proof of Corollary \ref{cor:2}] Assume that $q_0$ is real-valued and $\lambda_0 \in \Res(q_0) \cap i \R$. Let $\lambda_N$ be the random variable constructed in the proof of Theorem \ref{thm:2}. Then $\lambda_N$ is purely imaginary. Indeed, $\lambda_N|_{B_N} = \lambda_0 \in \R$; and on $\Omega_N \setminus B_N$, $\lambda_N$ is the unique resonance of $V_N$ in $\Dd(\lambda_0,r_0)$, in particular it is purely imaginary -- otherwise $-\overline{\lambda_N}$ would be another resonance of $V_N$ in the disk $\Dd(\lambda_0,r_0)$. 

The convergence results follows now from Theorem \ref{thm:2}, from the identity $g = \overline{f}$ and from the convergence mapping theorem \cite[Theorem 25.7]{Bi}. For instance, in Case I,
\begin{equation*}
\dfrac{N^{d/2}(\lambda_N-\lambda_0)}{i\int_{\R^d} q(x) dx} = \Re \left( \dfrac{N^{d/2}(\lambda_N-\lambda_0)}{i\int_{\R^d} q(x) dx} \right) = \pi\left( \dfrac{N^{d/2}(\lambda_N-\lambda_0)}{i\int_{\R^d} q(x) dx} \right) \law \NN\left(0,\sigma^2 \right).
\end{equation*} 
where a complex number $x+iy$ is seen as a vector $(x,y)$; $\pi(x + i y) = \pi(x,y) = x$; and $\sigma^2 = \int_{[-1,1]^d} |f(x)|^4 dx$. The statement about eigenvalues follows from \cite[p. 31]{DZ}: in the context of real-valued potentials, eigenvalues corresponds exactly to resonances on the half-line $i(0,\infty)$.
\end{proof}

\subsection{Necessity of the assumptions of Theorem \ref{thm:1}}\label{app:2}

We show on a simple explicit example that the conclusion of Theorem \ref{thm:1},
\begin{equation}\label{eq:1o}
N \text{ sufficiently large} \ \Rightarrow \ \Res(V_N) \cap \Dd(0,R) \subset \bigcup_{\lambda \in \Res(q_0)} \Dd\left(\lambda, N^{-\frac{\gamma}{2m_\lambda}} \right)
\end{equation}
cannot hold with probability $1$. 

We fix $d=1$, $q_0 \equiv 0$, $q \in C_0^\infty(\R,\R)$ with $\int_\R q(x) dx = 1$ and $u_j$ independent Bernouilli random variables ($\Pp(u_j=1) = \Pp(u_j = -1) = 1/2$). We observe that
\begin{equation*}
\Pp(\{ u_j=1 \ \forall j \in [-N,N]^3 \}) = 2^{-N} > 0,
\end{equation*}
and the potential corresponding to this event is $\tV_N(x) \de \sum_j q(Nx-j)$. The weak limit of $\tV_N$ as $N \rightarrow +\infty$ is $\1_{[-1,1]}$, and the convergence is in fact strong in $H^{-2}$. Indeed, if $\varphi \in C_0^\infty(\R,\C)$ then using a representation of $\int_{[-1,1]} \varphi(x) dx$ as a Riemann sum modulo $O(N^{-1} |\varphi'|_\infty)$,
\begin{equations*}
\lr{\tV_N,\varphi} - \lr{\1_{[-1,1]},\varphi} = \sum_j \int_{\R} q(Nx-j) \varphi(x) dx - \int_{[-1,1]} \varphi(x) dx \\
 = \dfrac{1}{N} \left( \sum_j \int_{\R} q(x) \varphi \left( \dfrac{x+j}{N} \right) dx - \sum_j \int_{\R} q(x) \varphi\left(\dfrac{j}{N} \right)dx\right) + O(N^{-1} |\varphi'|_\infty) \\
  = \dfrac{1}{N} \sum_j \int_\R q(x) \left(\varphi \left( \dfrac{x+j}{N} \right) - \varphi\left( \dfrac{j}{N} \right)\right) dx + O(N^{-1} |\varphi'|_\infty).
\end{equations*}
This is fully estimated by $O(N^{-1}|\varphi'|_\infty)$. Hence, $\tV_N-\1_{[-1,1]}$ converges to $0$ for the topology of distributions of order $1$ on $\R$. In dimension one, functions that are locally in $C^1$ are locally in $H^2$. Therefore $\tV_N-\1_{[-1,1]}$ also converges to $0$ for the topology of bounded linear functionals on $H^2$, i.e. for the $H^{-2}$ norm-topology, as claimed.

According to Lemma \ref{lem:1b}, resonances of $\tV_N$ converge to resonances of $\1_{[-1,1]}$ -- uniformly on compact sets. On the other hand, the potential $\1_{[-1,1]}$ has infinitely many resonances -- see for instance \cite[Theorem 2.14]{DZ} -- while $q_0$ has a single resonance. This shows that \eqref{eq:1o} cannot hold for every $R > 0$ when $\tV_N$ replaces $V_N$, in particular Theorem \ref{thm:1} holds with probability at most $1-2^{-N}$. 

We end up by mentionning that this construction can be adapted to higher dimension, with the conclusion that Theorem \ref{thm:1} cannot hold with probability greater than $1-2^{-N^d}$. The construction requires Smith--Zworski \cite{SZ} instead of \cite{DZ}: in odd dimension $d \geq 3$, bounded compactly supported real-valued potentials have at least one resonance.

\end{document}